\documentclass{ws-ijbc}
\usepackage{ws-rotating}     
\usepackage{graphicx}
\usepackage{epstopdf}
\usepackage{amsfonts, amsbsy, amssymb, amsmath, float, subfigure, natbib, color, url}

\begin{document}

\catchline{}{}{}{}{} 

\markboth{Lopesino et al.}{The Chaotic Saddle in the Lozi map, autonomous and non-autonomous versions.}

\title{\Large{\textbf{THE CHAOTIC SADDLE IN THE LOZI MAP, AUTONOMOUS AND NONAUTONOMOUS VERSIONS}}}



\author{CARLOS LOPESINO\footnote{\textit{Instituto de Ciencias Matem\'aticas, CSIC-UAM-UC3M-UCM, Madrid Spain, 
carlos.lopesino@icmat.es}}}
\address{Instituto de Ciencias Matem\'aticas, CSIC-UAM-UC3M-UCM, \\ C/ Nicol\'as Cabrera 15, Campus Cantoblanco UAM, 
28049\\Madrid, Spain\\ carlos.lopesino@icmat.es}

\author{FRANCISCO BALIBREA-INIESTA\footnote{\textit{Instituto de Ciencias Matem\'aticas, CSIC-UAM-UC3M-UCM, Madrid 
Spain, francisco.balibrea@icmat.es}}}
\address{Instituto de Ciencias Matem\'aticas, CSIC-UAM-UC3M-UCM, C/ Nicol\'as Cabrera 15, Campus Cantoblanco UAM, 
28049\\Madrid, Spain\\ francisco.balibrea@icmat.es}

\author{STEPHEN WIGGINS\footnote{\textit{School of Mathematics, University of Bristol, Bristol, United Kingdom, 
maxsw@bristol.ac.uk}}}
\address{School of Mathematics, University of Bristol BS8 1TW\\Bristol, United Kingdom\\ maxsw@bristol.ac.uk}

\author{ANA M. MANCHO\footnote{\textit{Instituto de Ciencias Matem\'aticas, CSIC-UAM-UC3M-UCM, Madrid 
Spain, a.m.mancho@icmat.es}}}
\address{Instituto de Ciencias Matem\'aticas, CSIC-UAM-UC3M-UCM, C/ Nicol\'as Cabrera 15, Campus Cantoblanco UAM, 
28049\\Madrid, Spain\\ a.m.mancho@icmat.es}

\maketitle

\begin{history}
\received{(to be inserted by publisher)}
\end{history}

\begin{abstract}
  In this paper we prove the existence of a chaotic saddle for a piecewise linear map of the plane, referred to as the 
Lozi map. We study the Lozi map in its orientation and area preserving version. First, we consider the autonomous 
version of the Lozi map to which we apply the Conley-Moser conditions to obtain the proof of a chaotic saddle. Then we 
generalize the Lozi map on a non-autonomous version and we prove that the first and the third Conley-Moser conditions 
are satisfied, which imply the existence of a chaotic saddle. Finally, we numerically demonstrate how the structure of 
this nonautonomous chaotic saddle varies as parameters are varied.
\end{abstract}

\keywords{Chaotic Saddle, autonomous dynamics, nonautonomous dynamics, Lozi map, Conley-Moser conditions.}

\section{Introduction}\label{sec:intro}

In this paper we prove that the Lozi map (\cite{Lozi}), as well as a nonautonomous generalization of the Lozi map, 
possesses a chaotic saddle, i.e. a hyperbolic invariant set on which the dynamics is topologically conjugate to a 
Bernoulli shift.  Our construction uses the Conley-Moser conditions for autonomous maps as developed in \cite{Moser} 
(see also \cite{Wiggins03}) and the nonautonomous Conley-Moser conditions as developed in \cite{Wiggins99} and 
\cite{balibrea}. For earlier work in a similar spirit as the Conley-Moser conditions see \cite{vma_a,vma_b,vma_c}.

Previously, the autonomous Conley-Moser conditions have been  used by \cite{Dev79} to show the existence of a chaotic invariant set for the H\'enon map  and by \cite{Holmes} and \cite{ChasBerGem} to show the existence of a chaotic invariant set for the bouncing ball map.

While the development of the ``dynamical systems approach to nonautonomous dynamics'' is currently a topic of much interest in the pure mathematics community, 
it is not a topic that is widely known in the applied dynamical systems community 
(especially the fundamental work that was done in the 1960's). An applied motivation for such work is an understanding from the dynamical systems point of view of fluid transport for aperiodically time dependent flows. \cite{wm} have given a survey of the history of nonautonomous dynamics as well as its application to fluid transport. Earlier work on chaos in nonautonomous systems is described in \cite{ls,stoffa,stoffb}.

This paper is outlined as follows. In section \ref{sec:intro} we introduce the setup of the problem. The  definitions and theorems given in this section make clear what we mean by the phrase “ chaotic invariant set ” for both autonomous and nonautonomous maps. In sections \ref{sec:Loziaut} and \ref{sec:Lozinaut} we construct chaotic invariant sets (i.e., {\em chaotic saddles}) for both  the autonomous and nonautonomous versions of the Lozi map,  respectively. In section \ref{sec:DLD} we show how these sets are detected using the Discrete  Lagrangian Descriptor (see \cite{Lopesino}) for different parameter values, both in the autonomous or in the nonautonomous case. Finally, in section \ref{sec:summ} we summarize our results.

\section{Setup and geometry of the problem}
\label{sec:setup}

 In this section we recall the set-up for the autonomous  Conley-Moser conditions that were introduced by \cite{Moser} and the nonautonomous Conley-Moser conditions introduced in \cite{Wiggins99}. We follow  the structure in
  \cite{Wiggins03} for our exposition, but with an inverse notation (that is $f \equiv L^{-1}$ in the autonomous case and $f_n \equiv 
  L_n^{-1}$ in the nonautonomous case, where we use the notation $L$ in the autonomous case and $L_n^{-1}$ in the nonautonomous case to denote the general form for  the maps under consideration, rather than $f$ and $f_n$, respectively, from the original references, since $L$ is traditionally used to refer to the Lozi map).  

  \subsection{The Autonomous Conley-Moser Conditions}

  We begin with the following two definitions.
  
  \begin{definition}
  A $\mu_v$-vertical curve is the graph of a function $x=v(y)$ for which
  $$-R \leq v(y) \leq R, \qquad |v(y_1)-v(y_2)| \leq \mu_v |y_1-y_2| \quad \text{for } -R \leq y_1, y_2 \leq R.$$
  Similarly, a $\mu_h$-horizontal curve is the graph of a function $y=h(x)$ for which
  $$-R \leq h(x) \leq R, \qquad |h(x_1)-h(x_2)| \leq \mu_h |x_1-x_2| \quad \text{for } -R \leq x_1, x_2 \leq R.$$
  \end{definition}
  
  \begin{definition}
  Given two nonintersecting $\mu_v$-vertical curves $v_1(y) < v_2(y)$, $y \in [-R,R]$, we define a $\mu_v$-vertical 
  strip as
  $$V = \left\lbrace (x,y) \in \mathbb{R}^2 | x \in [v_1(y),v_2(y)]; \quad y \in [-R,R] \right\rbrace .$$
  Similarly, given two nonintersecting $\mu_h$-horizontal curves $h_1(x) < h_2(x)$, $x \in [-R,R]$, we define a 
  $\mu_h$-horizontal strip as
  $$H = \left\lbrace (x,y) \in \mathbb{R}^2 | y \in [h_1(x),h_2(x)]; \quad x \in [-R,R] \right\rbrace .$$
  The width of horizontal and vertical strips is defined as
  $$d(H) = \max_{x\in[-R,R]} |h_2(x) - h_1(x)|,$$
  $$d(V) = \max_{y\in[-R,R]} |v_2(y) - v_1(y)|.$$
  \end{definition}  
  
  \par
  \noindent
  Keeping these definitions in mind, we begin with the Conley Moser conditions for the autonomous case. We consider a map
  
  $$L : D \longrightarrow \mathbb{R}^2,$$
  
  \noindent
 where $D$ is a square in $\mathbb{R}^2$, i.e.,
 
  $$D = \left\lbrace (x,y) \in \mathbb{R}^2 | -R \leq x \leq R, \quad -R \leq y \leq R \right\rbrace .$$

  \noindent
  Let
  
  $$I = \left\lbrace 1,2, ... ,N \right\rbrace , \qquad (N \geq 2),$$
  \noindent
  be an index set, and let
  
  $$H_i, \qquad i \in I$$
  
  \noindent
  be a set of disjoint $\mu_h$-horizontal strips and let
  
  $$V_i, \qquad i \in I$$
  \noindent
  be a set of disjoint $\mu_v$-vertical strips. Suppose that $L$ satisfies the following two conditions.
  \par

  \noindent
  \textbf{Assumption 1:} $0 \leq \mu_v \mu_h < 1$ and $L$ maps $V_i$ homeomorphically onto $H_i$, ($L(V_i) = H_i$) for 
  $i \in I$. Moreover, the horizontal boundaries of $V_i$ map to the horizontal boundaries of $H_i$ and the vertical 
  boundaries of $V_i$ map to the vertical boundaries of $H_i$.
  \par

  \noindent
  \textbf{Assumption 2:} Suppose $H$ is a $\mu_h$-horizontal strip contained in $\cup_{i \in I} H_i$. Then

  $$L(H) \cup H_i \equiv \tilde{H}_i$$
  
  \noindent
  is a $\mu_h$-horizontal strip for every $i \in I$. Moreover,

  $$d(\tilde{H}_i) \leq \nu_h d(H) \qquad \text{for some} \quad 0 < \nu_h < 1.$$
  \par

  \noindent
  Similarly, suppose $V$ is a $\mu_v$-vertical strip contained in $\cup_{i \in I}V_i$. Then
  
  $$L^{-1}(V) \cap V_i \equiv \tilde{V}_i$$
  
 \noindent
 is a $\mu_v$-vertical strip for every $i \in I$. Moreover,

  $$d(\tilde{V}_i) \leq \nu_v d(V) \qquad \text{for some} \quad 0 < \nu_v < 1.$$
  
  Then we have the following theorem.
  
  \begin{theorem}
  Suppose $L$ satisfies \textit{Assumptions} $1$ and $2$. Then $L$ has an invariant Cantor set, $\Lambda$, on which it 
  is topologically conjugate to a full shift on $N$ symbols, i.e., the following diagram commutes
  \begin{equation}
  \begin{array}{ccccc} 
  & & L & & \\ & \Lambda & \longrightarrow & \Lambda \\ \\ \phi & \downarrow & & \downarrow & \phi 
  \\ & & \sigma & & \\ & \Sigma^N & \longrightarrow & \Sigma^N & \\ 
  \end{array} 
  \end{equation}
  where $\phi$ is a homeomorphism mapping $\Lambda$ onto $\Sigma^N$ and $\sigma$ denotes the shift map acting on on space of bi-infinite sequence of $N$ symbols, denoted by $\Sigma^N$.
  \end{theorem}  

More details on the map $\sigma$ and the space $\Sigma^N$ can be found in \cite{Wiggins03}, as well as in comments following \eqref{eq:metric}.

\subsection{The Nonautonomous Conley-Moser Conditions}
  
  Next we describe the setting for  the nonautonomous Conley-Moser conditions. These conditions were generalized by \cite{Wiggins99} but we are using the notation used in \cite{balibrea},where a more detailed discussion of the definitions is given.
  
  \par 
  \noindent
  \begin{definition}
  Let $D \subset \mathbb{R}^{2}$ denote a closed and bounded set. We define its projections as
  
  $$D_x = \left\lbrace x \in \mathbb{R} \text{ for which there exists a } y \in \mathbb{R} \text{ with } (x,y) \in D 
\right\rbrace$$

  $$D_y=  \left\lbrace y \in \mathbb{R} \text{ for which there exists a } x \in \mathbb{R} \text{ with } (x,y) \in D 
\right\rbrace$$

  \end{definition}
  
  \noindent
  $D_x$ and $D_y$  represent the projections of $D$ onto the $x$ axis and the $y$ axis, respectively. Let $I_x$ be 
  a closed interval contained in $D_x$ and let $I_y$ be a closed interval contained in $D_y$.
  
  \begin{definition}
  Let $0 \leq \mu_h < \infty$. A $\mu_h$-horizontal curve, $\bar{H}$, is defined to be the graph of a function $h: I_x 
  \longrightarrow \mathbb{R}$ where $h$ satisfies the following two conditions:
  
  \begin{enumerate}
  \item The set $\bar{H} = \lbrace (x,h(x)) \in \mathbb{R} \times \mathbb{R} \text{ such that } x \in I_x \rbrace $ is 
  contained in $D$.
  \item For every $x_1$, $x_2 \in I_x$, we have
  
  \end{enumerate}
  
  \begin{equation}
  |h(x_1) - h(x_2)| \leq \mu_h |x_1 - x_2|.
  \end{equation}
  
  \noindent
  Similarly, let $0 \leq \mu_v < \infty$. A $\mu_v$-vertical curve, $\bar{V}$, is defined to be the graph of a 
  function $v: I_y \longrightarrow \mathbb{R}$ where $v$ satisfies the following two conditions:
  
  \begin{enumerate}
  \item The set $\bar{V} = \lbrace (v(y),y) \in \mathbb{R} \times \mathbb{R} \text{ such that } y \in I_y \rbrace $ is 
  contained in $D$.
  \item For every $y_1$, $y_2 \in I_y$, we have
  \end{enumerate}
  
  \begin{equation}
  |v(y_1) - v(y_2)| \leq \mu_v |y_1 - y_2|.
  \end{equation}
  
  \end{definition}
  
  \noindent
  Now we can define two dimensional strips by using  these horizontal and vertical curves.

  \begin{definition}
  Given two nonintersecting $\mu_v$-vertical curves $v_1(y) < v_2(y)$, $y \in I_y$, we define a $\mu_v$ vertical strip 
  as
  $$V = \lbrace (x,y) \in \mathbb{R}^2 \text{ such that } x \in [v_1(y), v_2(y)] \text{; } y \in I_y \rbrace .$$
  Similarly, given two nonintersecting $\mu_h$-horizontal curves $h_1(x) < h_2(x)$, $x \in I_x$, we define a $\mu_h$ 
  horizontal strip as
  
  $$H = \lbrace (x,y) \in \mathbb{R}^2 \text{ such that } y \in [h_1(x), h_2(x)] \text{; } x \in I_x \rbrace .$$
  
  \noindent
  The width of horizontal and vertical strips is defined as 
  
  $$d(H) = \max_{x\in I_x} |h_2(x) - h_1(x)|,$$
  
  $$d(V) = \max_{y\in I_y} |v_2(y) - v_1(y)|.$$
  \end{definition}
  
  \noindent
  Additionally, we  define horizontal and vertical boundaries of the strips.
  
  \begin{definition}
  The vertical boundary\footnote{The symbol $\partial$ is the usual notation from topology denoting the boundary of a set. In this paper we further refine this notion by referring to horizontal boundaries, $\partial_h$, and vertical boundaries, $\partial_v$.} of a $\mu_h$-horizontal strip $H$ is denoted $\partial_vH$ and  is defined as
  $$\partial_vH = \lbrace (x,y) \in H \text{ such that } x \in \partial I_x \rbrace.$$
  The horizontal boundary of a $\mu_h$-horizontal strip $H$ is denoted $\partial_hH$ and  is defined as
  $$\partial_hH \equiv \partial H - \partial_vH.$$
  We can similarly the boundaries of $\mu_v$-vertical strips.
  \end{definition}
  
  \noindent
  Now we need to define another kind of strip which will appear in the nonautonomous Conley-Moser  conditions.
  
  \begin{definition}
  We say that $H$ is a $\mu_h$-horizontal strip contained in a $\mu_v$-vertical strip $V$ if the two $\mu_h$-horizontal 
  curves defining the vertical boundary of $H$ are contained in $V$, with the remaining boundary components of $H$ 
  contained in $\partial_vV$.
  \end{definition}

  \noindent
  The two $\mu_h$-horizontal curves defining the horizontal boundary of $H$ are referred to as the horizontal boundary 
  of $H$, and the remaining boundary components are referred to as the vertical boundary of $H$. This is described in detail  in the next definition.

  \begin{definition}
  Let $\tilde{H}$ be a $\mu_h$-horizontal strip contained in a  $\mu_v$-vertical strip, $V$. We define the boundaries of 
  $\tilde{H}$ as
  $$\partial_v\tilde{H} = \lbrace (x,y) \in \partial \tilde{H} \text{ such that } (x,y) \in \partial_vH \rbrace = 
  \partial \tilde{H} \cap \partial_vH ,$$
  and
  $$\partial_h \tilde{H} = \partial \tilde{H} - \partial_v \tilde{H}.$$
  The boundaries of $\tilde{V} \mu_v$-vertical strip contained in $H \mu_h$-horizontal strip are defined analogously.  
  \end{definition}

  \noindent
  We will be interested in the behaviour of $\mu_v$-vertical strips under maps. We want to focus in the case when the 
  image of a $\mu_v$-vertical strip intersects its preimage. 
  
  \begin{definition}
  Let $V$ and $\tilde{V}$ be $\mu_v$-vertical strips. $\tilde{V}$ is said to intersect $V$ fully if $\tilde{V} 
  \subset V$ and $\partial_h \tilde{V} \subset \partial_hV$.
  \end{definition}

  Now we can state the main theorem for the nonauntonomous case. Let $\lbrace L_{n},D_{n} \rbrace_{n=-\infty}^{+ \infty}$ a sequence of maps with
  
  \begin{equation} 
  L_{n}: D_{n} \longrightarrow D_{n+1} \text{ ,} \quad \forall n \in \mathbb{Z} \quad \text{and} \quad 
  L^{-1}_{n}: D_{n+1} \longrightarrow D_{n}
  \end{equation}
  
  \noindent
  in case the corresponding inverse function exists.
  \newline
  \\
 We require that on each domain $D_{n}$ there exists a finite collection 
  of vertical strips $V_{i}^{n} \subset D_{n}$ ($\forall n \in \mathbb{Z}$ and $\forall i \in I = \lbrace 1,2,...,N 
  \rbrace $) which map into a finite collection of horizontal strips located in $D_{n+1}$:
  
  \begin{equation}
  H_{i}^{n+1} \subset D_{n+1} \quad \text{with} \quad L_{n}(V_{i}^{n})=H_{i}^{n+1} \text{ ,} \quad 
  \forall n \in \mathbb{Z} \text{ ,} \quad i\in I
  \end{equation}
  
  \noindent
  We also need to define
  
  \begin{equation}
  \begin{array}{rcl}
  H^{n+1}_{ij} & \equiv & H^{n+1}_i \cap V^{n+1}_j \\
               &        &              \\
  V^n_{ji}     & \equiv & L^{-1}_n (V^{n+1}_j) \cap V^n_i \\
  \end{array}
  \end{equation}

  Following this idea we introduce the definition of transition matrix associated to a sequence of maps $\lbrace 
  L_{n},D_{n} \rbrace_{n=-\infty}^{+ \infty}$,
  
  $$ A \equiv \lbrace A^{n} \rbrace_{n=-\infty}^{+\infty} \text{ is a sequence of matrices of dimension } N \times N 
  \text{ such that}$$
  
  $$ A_{ij}^{n}= \begin{cases} 1 \quad \quad \text{if } L_{n}(V_{i}^{n}) \cap V_{j}^{n+1} \not = \emptyset \\ 0 \quad 
  \quad \text{otherwise} \end{cases} \text{or equivalently} $$

  \begin{equation}
  A_{ij}^{n}= \begin{cases} 1 \quad \quad \text{if } H_{i}^{n+1} \cap V_{j}^{n+1} = H^{n+1}_{ij} \not = \emptyset \\ 0 
  \quad \quad \text{otherwise} \end{cases} \forall i,j \in I \hspace{1.65cm}
  \end{equation}
  \\
  which will be needed for applying the Conley-Moser conditions to a given sequence of maps $\lbrace L_{n},D_{n} 
  \rbrace_{n=-\infty}^{+ \infty}$ and then proving the existence of a chaotic invariant set.
  \newline
  \\
  \textbf{Assumption 1:} For all $i,j \in I$ such that $A_{ij}^{n}=1$, $H_{ij}^{n+1}$ is a $\mu_{h}$-horizontal strip 
  contained in $V_{j}^{n+1}$ with $0 \leq \mu_{v} \mu_{h} <1$. Moreover, $L_{n}$ maps $V_{ji}^{n}$ homeomorphically 
  onto $H_{ij}^{n+1}$ with $L_{n}^{-1}(\partial_{h}H_{ij}^{n+1}) \subset \partial_{h}V_{i}^{n}$.
  
  \begin{remark}
  The fact that every non empty $H_{ij}^{n+1} \subset D_{n+1}$ is a $\mu_{h}$-horizontal strip contained in 
  $V_{j}^{n+1}$ shows that the two $\mu_{h}$-horizontal curves which form the boundary ($\partial_{h} 
  L_{n}(V_{i}^{n}) = \partial_{h}H_{i}^{n+1}$) cut the horizontal boundary of $V_{i}^{n+1}$ in exactly four points.
  \end{remark}
  
  \noindent
  Furthermore, since $L_{n}$ is one-to-one on $D_{V}^{n} \equiv \cup_{i=1}^{N}V_{i}^{n}$ then we can define an 
  inverse function $L_{n}^{-1}$ on $L_{n}(D_{V}^{n})=\cup_{i=1}^{N}L_{n}(V_{i}^{n}) \equiv \cup_{i=1}^{N}H_{i}^{n+1}$.
  \newline
  \\
  And since $L_{n}$ maps $V_{ji}^{n}$ homeomorphically onto $H_{ij}^{n+1}$ with $L_{n}^{-1}(\partial_{h}H_{ij}^{n+1}) 
  \subset \partial_{h}V_{i}^{n}$ then $L_{n}^{-1}$ maps $H_{ij}^{n+1}$ homeomorphically onto $V_{ji}^{n}$ ($\forall i,j 
  \in I$) with
  
  \begin{equation} 
  L_{n} \left( L_{n}^{-1}(\partial_{h}H_{ij}^{n+1}) \right) =  \partial_{h}H_{ij}^{n+1} \subset L_{n}( \partial_{h} 
V_{i}^{n})
  \end{equation}
  
  \noindent
  \textbf{Assumption 2:} Let $V^{n+1}$ be a $\mu_{v}$-vertical strip which intersects $V_{j}^{n+1}$ fully. Then 
  $L_{n}^{-1}(V^{n+1}) \cap V_{i}^{n} \equiv \widetilde{V}_{i}^{n}$ is a $\mu_{v}$-vertical strip intersecting 
  $V_{i}^{n}$ fully for all $i \in I$ such that $A_{ij}^{n}=1$. Moreover,
  
  \begin{equation} 
  d(\widetilde{V}_{i}^{n}) \leq \nu_{v} \text{ } d(V^{n+1}) \quad \quad \text{for some } 0< \nu_{v} < 1 
  \end{equation}
  \\
  Similarly, let $H^{n}$ be a $\mu_{h}$-horizontal strip contained in $V_{i}^{n}$ such that also $H^{n} \subset 
  H_{ji}^{n}$ for some $i,j \in I$ with $A_{ji}^{n-1}=1$. Then $L_{n}(H^{n}) \cap V_{k}^{n+1} \equiv 
  \widetilde{H}_{k}^{n+1}$ is a $\mu_{h}$-horizontal strip contained in $V_{k}^{n+1}$ for all $k \in I$ such that 
  $A_{ik}^{n}=1$. Moreover,
  
  \begin{equation} 
  d(\widetilde{H}_{k}^{n+1}) \leq \nu_{h} \text{ } d(H^{n}) \quad \quad \text{for some } 0< \nu_{h} < 1.
\label{eq:metric} 
  \end{equation}

  We also need to adapt some definitions from symbolic dynamics. Let
  
  \begin{equation} 
  s=( \cdots s_{n-k} \cdots s_{n-2} s_{n-1}. s_{n} s_{n+1} \cdots s_{n+k} \cdots ) 
  \end{equation}
  denote a bi-infinite sequence with $s_{l} \in I$ ($\forall l\in \mathbb{Z}$) where adjacent elements of the sequence 
  satisfy the rule $A_{s_{n}s_{n+1}}^{n}=1$, $\forall n \in \mathbb{Z}$.
  \newline
  \\
We denote the set of all such symbol 
  sequences by $\Sigma_{ \lbrace A^{n} \rbrace }^{N}$. If $\sigma$ denotes the shift map
  \begin{equation} 
  \sigma (s)= \sigma ( \cdots s_{n-2} s_{n-1}. s_{n} s_{n+1} \cdots ) =  ( \cdots s_{n-2} s_{n-1} 
  s_{n}. s_{n+1} \cdots )
  \end{equation}
  on $\Sigma_{ \lbrace A^{n} \rbrace }^{N}$, we define the `extended shift map' $\tilde{\sigma}$ on $\widetilde{\Sigma} 
  \equiv \Sigma_{ \lbrace A^{n} \rbrace }^{N} \times \mathbb{Z}$ by
  
  \begin{equation} 
  \tilde{\sigma}(s,n)=(\sigma (s),n+1). \text{ It is also defined } f(x,y;n)=(L_{n}(x,y),n+1). 
  \end{equation}

  \noindent
  We now can  state the main theorem.
  \newline
  \\
  \begin{theorem}\label{main_nonautonomous_theorem}
  Suppose $\lbrace L_{n},D_{n} \rbrace_{n=-\infty}^{+\infty}$ satisfies A1 and A2. There exists a sequence of sets 
  $\Lambda_{n} \subset D_{n}$, with $L_{n}(\Lambda_{n})=\Lambda_{n+1}$, such that the following diagram commutes
  \begin{equation}
  \begin{array}{ccccc} 
  & & f & & \\ & \Lambda_{n} \times \mathbb{Z} & \longrightarrow & \Lambda_{n+1} 
  \times \mathbb{Z} & \\ \\ \phi & \downarrow & & \downarrow & \phi \\ & & \tilde{\sigma} & & \\ & \Sigma_{ \lbrace 
  A^{n} 
  \rbrace }^{N} \times \mathbb{Z} & \longrightarrow & \Sigma_{ \lbrace A^{n} \rbrace }^{N} \times \mathbb{Z} & 
  \end{array} 
  \end{equation}
  where $\phi (x,y;n) \equiv (\phi_{n}(x,y),n)$ with $\phi_{n}(x,y)$ a homeomorphism mapping $\Lambda_{n}$ onto 
  $\Sigma_{\lbrace A^{n} \rbrace }^N$.
  \end{theorem}

  \begin{remark}
  We are referring to $\left\lbrace \Lambda_n \right\rbrace$ as an infinite sequence of chaotic invariant sets. Defining
  $$\Lambda \equiv \cup_{n \in \mathbb{Z}} \Lambda_n,$$
  then $\Lambda$ contains an uncontable infinity of orbits where each orbit is unstable (of saddle type), and the 
  dynamics on the invariant set exhibits sensitive dependence on initial conditions.
  \end{remark}
  
  \noindent
  As we will see, it can be difficult to verify  Assumption 2 in specific examples. For that reason we will define the third condition 
  known as the 'cone condition'. Before stating this condition, we define the following
  
  \begin{equation} { \cal V}^{n} = \bigcup_{i,j \in I}V_{ij}^{n} \equiv \bigcup_{i,j \in I}V_{i}^{n} \cap 
  L_{n}^{-1}(V_{j}^{n+1}),
  \end{equation}
  
  \begin{equation} { \cal H}^{n+1} = \bigcup_{i,j \in I}H_{ji}^{n+1} \equiv \bigcup_{i,j \in I} V_{j}^{n+1} \cap 
  L_{n}(V_{i}^{n}) , \quad L_{n}({ \cal V}^{n})={ \cal H}^{n+1}
  \end{equation}
  
  \begin{equation} S^{s}_{{ \cal K}} = \lbrace (\xi_{z},\eta_{z}) \in \mathbb{R}^{2} \text{ } | \text{ } | \eta_{z} | 
  \leq \mu_{v} | \xi_{z} | \text{, } z \in { \cal K} \rbrace
  \end{equation}
  
  \begin{equation} S^{u}_{{ \cal K}} = \lbrace (\xi_{z},\eta_{z}) \in \mathbb{R}^{2} \text{ } | \text{ } | \xi_{z} | 
  \leq \mu_{h} | \eta_{z} | \text{, } z \in { \cal K} \rbrace,
  \end{equation}

\noindent
with $\cal K$ being either ${\cal V}^n$ or ${\cal H}^{n+1}$.
  
  Now we can state assumption 3.
  \newline
  \\
  \textbf{Assumption 3: The cone condition.} $Df_{n}(S^{u}_{{ \cal V}^{n}})\subset S^{u}_{ { \cal H}^{n+1}}$, 
  $Df^{-1}_{n}(S^{s}_{{ \cal H}^{n+1}})\subset S^{s}_{{ \cal V}^{n}}$.
  \\
  Moreover, if $(\xi_{L_{n}(z_{0}^{n})}, \eta_{L_{n}(z_{0}^{n})}) \equiv DL_{n}(z_{0}^{n}) \cdot (\xi_{z_{0}^{n}}, 
  \eta_{z_{0}^{n}}) \in S^{u}_{{ \cal H}^{n+1}}$ then 
  
  \begin{equation} 
  | \eta_{L_{n}(z_{0}^{n})} | \geq \left( \frac{1}{\mu} \right) | \eta_{z_{0}^{n}} |
  \end{equation}
  
  \noindent
  If $(\xi_{L_{n}^{-1}(z_{0}^{n+1})}, \eta_{L_{n}^{-1}(z_{0}^{n+1})}) \equiv DL_{n}^{-1}(z_{0}^{n+1}) \cdot 
  (\xi_{z_{0}^{n+1}}, \eta_{z_{0}^{n+1}}) \in S^{s}_{{ \cal V}^{n}}$ then 
  
  \begin{equation} 
  | \xi_{L_{n}^{-1}(z_{0}^{n+1})} | \geq \left( \frac{1}{\mu} \right) | \xi_{z_{0}^{n+1}} | \quad \text{for } \mu > 0 
  \end{equation}
  
  We have the following theorem.
  
  \begin{theorem} 
  If nonautonomous A1 and A3 are satisfied for $0< \mu < 1-\mu_{h}\mu_{v}$ then A2 is satisfied and so the Theorem 
  \ref{main_nonautonomous_theorem} holds.
  \end{theorem}
  
  \begin{proof}
  The proof can be found in \cite{balibrea}.
  \end{proof}  

\section{The  autonomous Lozi map}
\label{sec:Loziaut}

In this section we show that assumptions 1 and 2 of the Conley-Moser conditions hold for the autonomous Lozi map. Typically, it has been difficult to show that assumption 2 holds for specific maps. However, the relatively simple form of the Lozi map allows us to demonstrate assumption 2 explicitly.

  The (autonomous) Lozi  map is defined as $L(x,y) = (1 + y - a|x|, bx)$, where $a,b\in \mathbb{R}$:

  \begin{equation}
  L(x,y) = (1 + y - a|x|, \quad -x)
  \end{equation}
  
  \noindent
  This map is invertible with inverse

  \begin{equation}
  L^{-1}(x,y) = (-y, \quad x + a|y| - 1). 
  \end{equation}
  
  \noindent
 For our purposes we will consider the situation where the Lozi map is orientation and area preserving, which occurs  when $b=-1$. Henceforth, $b$ will be fixed at this value and we will view $a$ as a parameter.
  
The setup for the autonomous problem is as follows. We consider the map
  
  $$L:S \longmapsto \mathbb{R}^2,$$
  
  \noindent
  where the square $S = \left \{ (x,y) \in \mathbb{R}^2 : |x| \leq R, |y| \leq R \right \}$ and the boundaries of the square are
  
  $$\begin{array}{ll}
  L_1 = \left \{ (x,y) \in \mathbb{R}^2 | y = R \right \}, & L_2 = \left \{ (x,y) \in \mathbb{R}^2 | y =
  -R \right \}, \\
         &                                                    \\
  L_3 = \left \{ (x,y) \in \mathbb{R}^2 | x = R \right \}, & L_4 = \left \{ (x,y) \in \mathbb{R}^2 | x
  = -R \right \}. \\  
  \end{array}$$
  
  \noindent
  Furthermore, we need to define the subsets of the boundaries of the square $S$ as follows:
  
  $$\begin{array}{ll}
  L^+_1 = \left \{ (x,y) \in \mathbb{R}^2 | x \geq 0, y = R \right \}, & L^-_1 = \left \{ (x,y) \in
  \mathbb{R}^2 | x < 0, y = R \right \}, \\
         &                                                    \\
  L^+_2 = \left \{ (x,y) \in \mathbb{R}^2 | x \geq 0, y = -R \right \}, & L^-_2 = \left \{ (x,y) \in
  \mathbb{R}^2 | x < 0, y = -R \right \},  \\
         &                                                    \\
  L^+_3 = \left \{ (x,y) \in \mathbb{R}^2 | x = R, y \geq 0 \right \}, & L^-_3 = \left \{ (x,y) \in
  \mathbb{R}^2 | x = R, y < 0 \right \},\\
         &                                                    \\
  L^+_4 = \left \{ (x,y) \in \mathbb{R}^2 | x = -R, y \geq 0 \right \}, & L^-_4 = \left \{ (x,y) \in
  \mathbb{R}^2 | x = -R, y < 0 \right \},\\         
  \end{array}$$
  
  \noindent
  as shown  in Figure \ref{Square}.
  
  \begin{figure}[ht]
    \centering
    \includegraphics[scale=0.6]{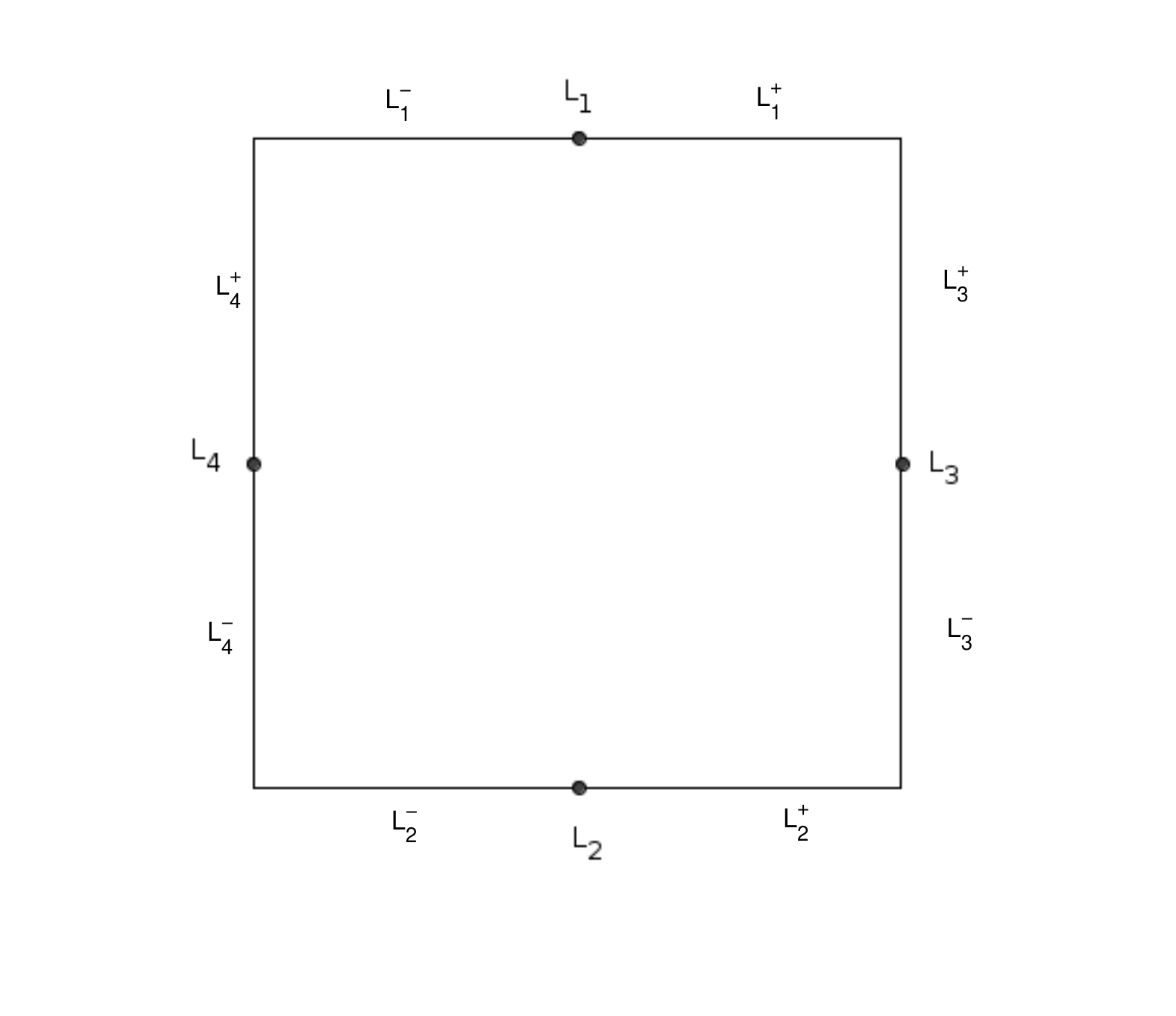}
    \caption{Geometrical setting: Boundaries and subsets of the boundaries of the square $S$.}
    \label{Square}
  \end{figure}
  
  At this moment, we can state the main result of this section; the existence of the chaotic saddle for the Lozi 
  autonomous map.

  \begin{theorem}\label{autonomous_main_theorem}
  For $a>4$  assumptions 1 and 2 hold for the  autonomous  Lozi map, and therefore  it possesses a chaotic saddle inside the square $S$.
  \end{theorem}

The proof of this theorem is carried out in the following subsections.

\subsection{Verification of Assumption 1 of the Autonomous Conley-Moser Conditions}
\label{sec:CM1}

  In this section we will show how the map $L$ acts on the square $S$. In 
  order to do that, we want to see how the boundaries of the square $S$ are transformed. This action can be seen in 
  figure \ref{Lozi_L}.
  \par
  First, we begin with the map $L$ acting on the upper boundary, $L_1$\footnote{For notational convenience henceforth we will denote the image of  a point $(x, y)$ under $L$ by $(X, Y)$. }.
  
  $$L(L_1) = L(x,R) = (1+R-a|x|,-x)$$

  $$\left \{ \begin{array}{ccc}
      X & = & 1+R-a|x| \\
      Y & = & -x \\
     \end{array}\right.
  $$
  \begin{equation}
  |y| = \frac{1+R-x}{a} \Rightarrow \left \{\begin{array}{cc}
                                              y = \frac{1+R}{a} - \frac{x}{a}, & x<0 \quad (L^-_1) \\
                                                                               &                   \\
                                              y = \frac{x}{a} - \frac{1+R}{a}, & x\geq0 \quad (L^+_1) \\
                                              \end{array} \right. 
  \end{equation}
  We do the same but for the lower boundary, $L_2$.
  \par
  \vspace{3mm}
  $$L(L_2) = L(x,-R) = (1-R-a|x|,-x)$$
  $$\left \{ \begin{array}{ccc}
      X & = & 1-R-a|x| \\
      Y & = & -x \\
     \end{array}\right.
  $$
  \begin{equation}
  |y| = \frac{1-R-x}{a} \Rightarrow \left \{\begin{array}{cc}
                                              y = \frac{1-R}{a} - \frac{x}{a}, & x<0 \quad (L^-_2) \\
                                                                               &                   \\
                                              y = \frac{x}{a} - \frac{1-R}{a}, & x\geq0 \quad (L^+_2) \\
                                              \end{array} \right.
  \end{equation}
  Now we see the map $L$ acting on the right side, $L_3$.
  \par
  \vspace{3mm}
  $$L(L_3) = L(R,y) = (1+y-aR,-R)$$
  \begin{equation}
  \left \{ \begin{array}{ccc}
      X & = & 1+y-aR \\
      Y & = & -R \\
      \end{array}\right. 
  \end{equation}  
  and the left side, $L_4$.
  \par
  \vspace{3mm}
  $$L(L_4) = L(-R,y) = (1+y-aR,R)$$
  \begin{equation}
  \left \{ \begin{array}{ccc}
      X & = & 1+y-aR \\
      Y & = & R \\
      \end{array}\right.
  \end{equation}
  From these computations we can set that the horizontal boundaries $L_1$ and $L_2$ are mapped by $L$ to affine lines 
  with slopes  $|m|=1/a$ passing through the points $q = (1+R,0)$ and $p = (1-R,0)$ respectively that must be outside
the domain $S$ as we will see later. Moreover, vertical lines are mapped to horizontal lines. From figure   
  \ref{Lozi_L} we can observe how the strips are formed. In this case the set $L(S) \cap S$ is the union of the two   
  horizontal strips $H_1$ and $H_2$. 

  \begin{figure}[ht]
    \centering
    \includegraphics[scale=0.75]{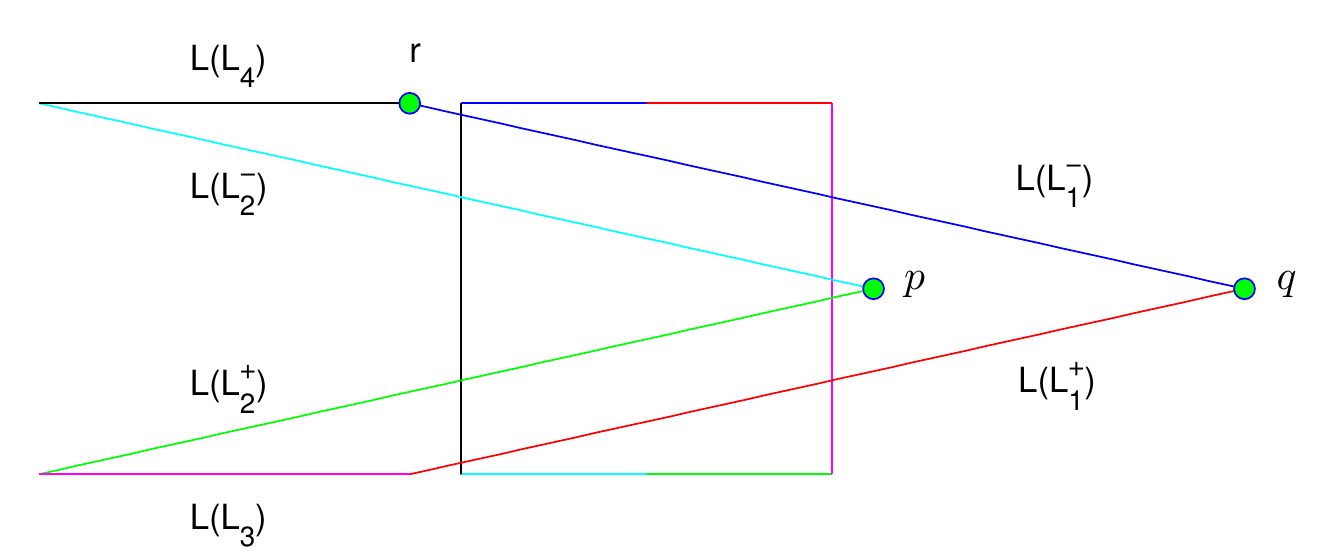}
    \caption{$L$ mapping $S$.}
    \label{Lozi_L}
  \end{figure}

  \par
  \noindent
  On the other hand, we want to show $L^{-1}$ acting on $S$, as is shown in figure \ref{Lozi_Lmen}. As we did 
  before, we start with $L^{-1}$ acting on the   upper side $L_1$.
  $$L^{-1}(L_1) = L^{-1}(x,R) = (-R,x+aR-1)$$
  \begin{equation}
  \left \{ \begin{array}{ccc}
      X & = & -R \\
      Y & = & x+aR-1 \\
      \end{array}\right. 
  \end{equation}
  and $L^{-1}$ acting on the lower side, $L_2$.
  \par
  \vspace{3mm}
  $$L^{-1}(L_2) = L^{-1}(x,-R) = (R,x+aR-1)$$
  \begin{equation}
  \left \{ \begin{array}{ccc}
      X & = & R \\
      Y & = & x+aR-1 \\
      \end{array}\right.
  \end{equation}
  Now we focus on $L^{-1}$ acting on the right side, $L_3$.
  \par
  \vspace{3mm}
  $$L^{-1}(L_3) = L^{-1}(R,y) = (-y,R+a|y|-1)$$
  $$\left \{ \begin{array}{ccc}
      X & = & -y \\
      Y & = & R+a|y|-1 \\
      \end{array}\right.
  $$
  \begin{equation}
  y = R+a|x|-1 \Rightarrow \left \{ \begin{array}{cc}
                                            y = R+ax-1, & y<0 \quad (L^-_3) \\
                                                        &                      \\
                                            y = R-ax-1, & y\geq0 \quad (L^+_3)    \\
                                      \end{array} \right.
  \end{equation}
  and $L^{-1}$ acting on the left side, $L_4$.
  \par
  \vspace{3mm}
  $$L^{-1}(L_4) = L^{-1}(-R,y) = (-y,-R+a|y|-1)$$
  $$\left \{ \begin{array}{ccc}
      X & = & -y \\
      Y & = & -R+a|y|-1 \\
      \end{array}\right.
  $$
  \begin{equation}
  y = a|x|-(1+R) \Rightarrow \left \{ \begin{array}{cc}
                                            y = ax-(1+R),  & y<0 \quad (L^-_4) \\
                                                           &                      \\
                                            y = -ax-(1+R), & y\geq0 \quad (L^+_4)    \\
                                        \end{array} \right.
  \end{equation}
  From this we can observe that the vertical boundaries $L_3$ and $L_4$ are mapped by $L^{-1}$ to affine lines with 
  slopes  $|m|=a$ passing through the points $\tilde{p} = (0,R-1)$ and $\tilde{q} = (0,-(1+R))$ respectively, 
  which are outside the square, and horizontal lines are mapped to vertical lines. As we show in figure 
  \ref{Lozi_Lmen}, the union of the vertical strips $V_1$ and $V_2$ is the set $L^{-1}(S) \cap S$.

  \begin{figure}[ht]
    \centering
    \includegraphics[scale=0.75]{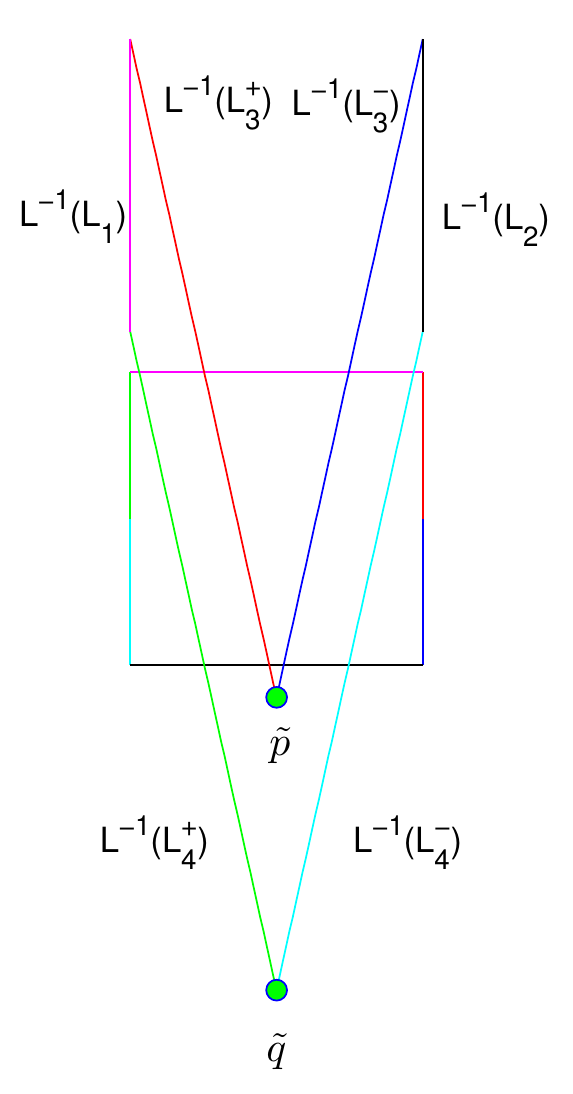}
    \caption{$L^{-1}$ mapping $S$.}
    \label{Lozi_Lmen}
  \end{figure}

  \par
  \noindent
  Now that we know how the square $S$ is mapped under $L$ and $L^{-1}$ we have to set some conditions on the size of 
  $S$, that is the value $R$. We have to keep in mind two conditions.
  \par
  \noindent
  The first one is that we need the two horizontal strips and the two vertical strips to  cross each other. (We 
  illustrate this idea in figure \ref{Lozi_L}). For that reason, we need the point $p$ to have first coordinate 
  greater than $R$ and the point $\tilde{p}$ to have second coordinate lower than $-R$. Because of the symmetry of the 
  problem, we reflect only the condition for $p$ in the next inequality

  \begin{equation}\label{first_condition}
  1-R \geq R \Leftrightarrow \frac{1}{2} \geq R
  \end{equation}
  \par
  \noindent
  On the other hand, the second issue is that we need the area of intersection between the strips, represented for instance in figure \ref{Lozi_cross}, to be inside the domain $S$. This is achieved when the point $r = L(L_1^{-}) \cap \lbrace y = R \rbrace$ has first coordinate lower than $-R$, as we can see in figure \ref{Lozi_L}.
  \begin{equation}
  r = L(L_1^{-}) \cap \lbrace y = R \rbrace = (1 + R - a|x|, -x)|_{-x=R} = (1 + R - aR, R)
  \end{equation}
  So this last assumption is translated into
  
  \begin{equation}\label{second_condition}
  1 + R -aR < -R
  \end{equation}
  
  \noindent
  and this holds when
  
  \begin{equation}
  R \geq \frac{1}{a-2}
  \end{equation}

  \noindent 
  Combining these two conditions we have,
  
  \begin{equation}
  \frac{1}{a-2} \leq R \leq \frac{1}{2}
  \end{equation}
  
  \noindent
  This last inequality only makes sense when $a \geq 4$.

  \par
  \noindent
  Once we have set these conditions, we can take
  
  \begin{equation}
  R(a) = \frac{1}{2}\left (\frac{1}{a-2} + \frac{1}{2}\right ) = \frac{a}{4(a-2)}, \quad \text{for every } a \geq 4,
  \end{equation}
  
  \noindent
  which is the midpoint of the interval
  
  $$\left [\frac{1}{a-2}, \quad \frac{1}{2} \right ]$$
  \par
  \noindent
  Now we start proving the first assumption, $A_1$, for the autonomous Conley-Moser conditions. For this purpose we need to define the vertical ($V_1$ and $V_2$) and 
the horizontal ($H_1$ and $H_2$) strips. To obtain the horizontal strips we will map forward by $L$ the square $S$ and 
make the intersection with the same square. To get the vertical strips we will map backward by $L^{-1}$ the square $S$ 
and intersect it with $S$, that is

  \begin{equation}
  \begin{array}{ccc}
  H_1 \cup H_2 & := & L(S) \cap S       \\
               &    &                   \\
  V_1 \cup V_2 & := & L^{-1}(S) \cap S  \\
  \end{array}
  \end{equation}
  
  \noindent
where $H_1$ is the upper half part of $L(S) \cap S$ and $H_2$ is the lower half part of $L(S) \cap S$, $V_1$ is the 
left half part of $L^{-1}(S) \cap S$ and $V_2$ is the right half part of $L^{-1}(S) \cap S$. We will use the following 
notation

  \begin{equation}
  \begin{array}{ccc}
  H_1 & := & (L(S) \cap S)^{+}\text{, where } y > 0, \\
      &    &                   \\
  H_2 & := & (L(S) \cap S)^{-}\text{, where } y < 0, \\
      &    &                   \\
  V_1 & := & (L^{-1}(S) \cap S)^{-}\text{, where } x < 0, \\
      &    &                   \\
  V_2 & := & (L^{-1}(S) \cap S)^{+}\text{, where } x > 0, \\
  \end{array}\label{eq30}
  \end{equation}
  These strips can be seen in figure \ref{Lozi_cross}.
  
  \begin{figure}[ht]
    \centering
    \includegraphics[scale=0.6]{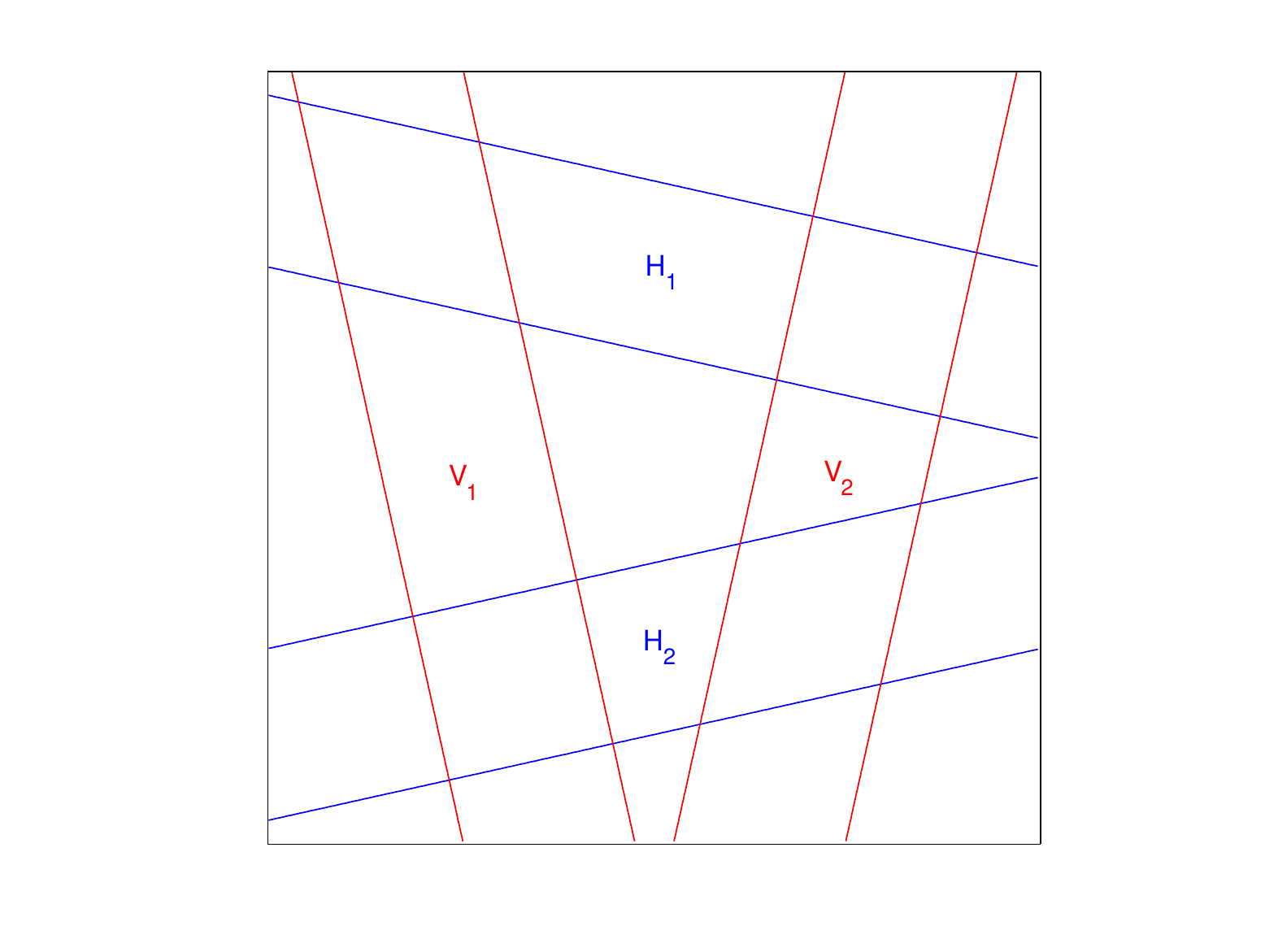}
    \caption{Vertical and horizontal strips. \textcolor{blue}{$H_1 \cup H_2 := L(S) \cap S$} and \textcolor{red}{$V_1 
\cup V_2 := L^{-1}(S) \cap S$}.}
    \label{Lozi_cross}
  \end{figure}

 It is easy to see from expressions in (\ref{eq30}) that $L$ maps $V_i$ homeomorphically onto $H_i$, ($L(V_i) = H_i$) for $i=1,2$. Furthermore, the horizontal boundaries of $V_i$ map to the horizontal boundaries of $H_i$ and the vertical boundaries of $V_i$ map to the vertical boundaries of $H_i$. Moreover, $V_i$ are $\mu_v$-vertical strips because its vertical boundaries are $\mu_v$-vertical curves where $|\mu_v| = 1/a$. And $H_i$ are $\mu_h$-horizontal strips because its horizontal boundaries are $\mu_h$-horizontal curves where $|\mu_h| = 1/a$.

  \begin{remark}
  We have to take care about $|\mu_v| = 1/a$. The slope of the vertical lines is $a$ but seen as horizontal curves. We 
must see this slope rate as if the vertical lines were $\mu_v$-vertical curves and so $|\mu_v| = 1/a$.
  \end{remark}  
  
  We have to keep in mind that the product of the slopes of the strips has to be less than 1 as it is required in the 
Conley-Moser conditions. In our case $|\mu_h| = |\mu_v| = 1/a$ Therefore

  \begin{equation}
  |\mu_h \cdot \mu_v| = |\frac{1}{a} \cdot \frac{1}{a}| = \frac{1}{a^2}, 
  \end{equation}

  \noindent
  and the condition of the slopes of $A_1$ is satisfied.

\subsection{Verification of Assumption 2 of the Autonomous Conley-Moser Conditions.}
\label{sec:CM2}

  The second assumption describes  how the map $L$ bends and makes thinner strips on each iteration. Moreover we obtain this rate which in our assumption is called $\nu_v$. Lets take $V$ a $\mu_v$-vertical strip contained, for instance, in $V_2$ ($V \subset V_2$). As we can see in figure \ref{Lozi_V_2}, we name the vertices of $V_2$ as $v_1=(v_{1_x},R)$, $v_2=(v_{2_x},R)$, $v_3=(v_{3_x},-R)$ and $v_4=(v_{4_x},-R)$ and the vertices which delimit $V$ are $q_1=(q_{1_x},R)$, $q_2=(q_{2_x},R)$, $q_3=(q_{3_x},-R)$ and $q_4=(q_{4_x},-R)$. We assume that
  \begin{equation}
      \begin{array}{ccccccc}
      v_{1_x} & < & q_{1_x} & < & q_{2_x} & < & v_{2_x} \\
      v_{3_x} & < & q_{3_x} & < & q_{4_x} & < & v_{4_x} \\
      \end{array}
  \end{equation}
  and
  $$d(V) = |q_{2_x}-q_{1_x}| = |q_{4_x}-q_{3_x}|$$
  The vertical boundaries of $V$ are the straight lines with slope rate $a$ (because $V$ is $\mu_v$-vertical strip) 
and pass through the points $q_1$ and $q_2$ respectively:
  \begin{equation}
    \begin{array}{l}
        y_{q_1} = ax+R-aq_{1_x} \\
        y_{q_2} = ax+R-aq_{2_x} \\
    \end{array}
  \end{equation}

  \begin{figure}[ht]
    \centering
    \includegraphics[scale=0.6]{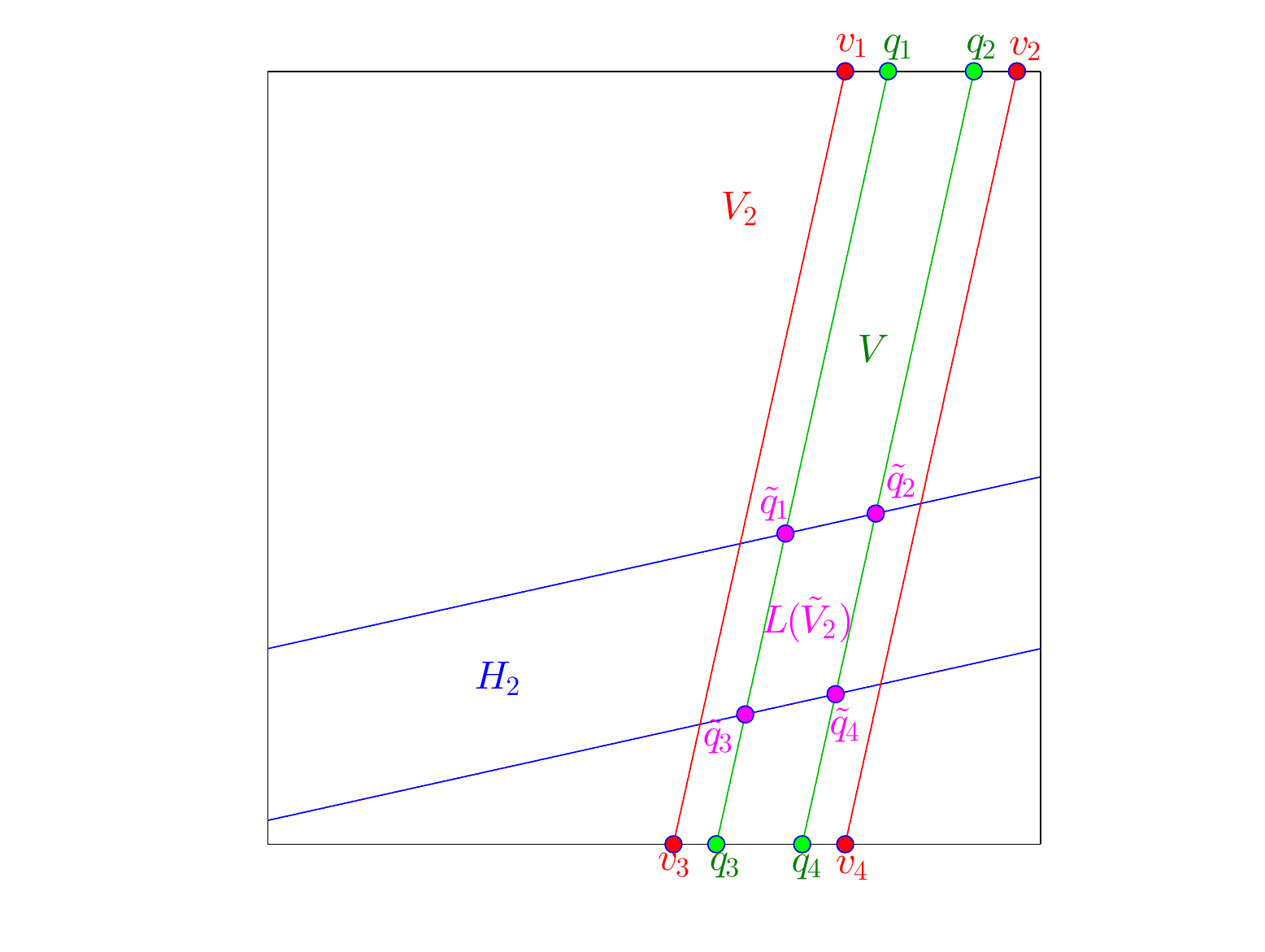}
    \caption{$V \subset V_2$ strip and $L(\tilde{V}_2) \equiv V \cap H_2$.}
    \label{Lozi_V_2}
  \end{figure}

  \noindent
  We need to determine $\tilde{V}_2 \equiv L^{-1}(V) \cap V_2$. Since we do not know how $L^{-1}$ acts on $V$, first we obtain 
$L(\tilde{V}_2) \equiv V \cap L(V_2) = V \cap H_2$ and after that, we will recover $\tilde{V}_2$ by iterating 
$L(\tilde{V}_2)$ backward. $V$ cuts $H_2$ at four points: $\tilde{q}_1$, $\tilde{q}_2$, $\tilde{q}_3$ and $\tilde{q}_4$ 
and we can describe them as follows:

  \begin{equation}
  \tilde{q}_1 : y_{q_1} \cap \left \{ y = \frac{x-(1-R)}{a} \right \} \Rightarrow \left \{
                                                                                   \begin{array}{l}
                                                                                     y = \frac{x-(1-R)}{a} \\
                                                                                     \\
                                                                                     y = ax +R - aq_{1_x}  \\
                                                                                   \end{array}
                                                                                   \right .
  \end{equation}
  
  \noindent
  Solving these two equations gives the coordinates of $\tilde{q}_1$:

  \begin{equation}
  \begin{array}{l}
   \tilde{q}_{1_x} = \frac{a^2q_{1_x}-aR-1+R}{a^2-1} \\
                                                             \\
   \tilde{q}_{1_y} = \frac{aq_{1_x}+aR-a-R}{a^2-1}
  \end{array}
  \end{equation}
  
  \noindent
  We continue the same procedure for the rest of the points

  \begin{equation}
   \tilde{q}_2 : y_{q_2} \cap \left \{ y = \frac{x-(1-R)}{a} \right \} \Rightarrow \left \{
                                                                                   \begin{array}{l}
                                                                                     y = \frac{x-(1-R)}{a} \\
                                                                                     \\
                                                                                     y = ax +R - aq_{2_x}  \\
                                                                                   \end{array}
                                                                                   \right .
  \end{equation}
  
  \noindent
  and the coordinates for $\tilde{q}_2$ are
  
  \begin{equation}
  \begin{array}{l}
   \tilde{q}_{2_x} = \frac{a^2q_{2_x}-aR-1+R}{a^2-1} \\
                                                             \\
   \tilde{q}_{2_y} = \frac{aq_{2_x}+aR-a-R}{a^2-1}
  \end{array}
  \end{equation}
  
  \noindent
  In the case of $\tilde{q}_3$ we solve
  
  \begin{equation}
   \tilde{q}_3 : y_{q_1} \cap \left \{ y = \frac{x-(1+R)}{a} \right \} \Rightarrow \left \{
                                                                                   \begin{array}{l}
                                                                                     y = \frac{x-(1+R)}{a} \\
                                                                                     \\
                                                                                     y = ax +R - aq_{1_x}  \\
                                                                                   \end{array}
                                                                                   \right .
  \end{equation}
  
  \noindent
  with coordinates
  
  \begin{equation}
  \begin{array}{l}
   \tilde{q}_{3_x} = \frac{a^2q_{1_x}-aR-1-R}{a^2-1} \\
                                                             \\
   \tilde{q}_{3_y} = \frac{aq_{1_x}-aR-a-R}{a^2-1}
  \end{array}
  \end{equation}
  
  \noindent
  Finally, to obtain $\tilde{q}_4$ we solve the system

  \begin{equation}
   \tilde{q}_4 : y_{q_2} \cap \left \{ y = \frac{x-(1+R)}{a} \right \} \Rightarrow \left \{
                                                                                   \begin{array}{l}
                                                                                     y = \frac{x-(1+R)}{a} \\
                                                                                     \\
                                                                                     y = ax +R - aq_{2_x}  \\
                                                                                   \end{array}
                                                                                   \right .
  \end{equation}
  
  \noindent
  therefore its coordinates are
  
  \begin{equation}
  \begin{array}{l}
   \tilde{q}_{4_x} = \frac{a^2q_{2_x}-aR-1-R}{a^2-1} \\
                                                             \\
   \tilde{q}_{4_y} = \frac{aq_{2_x}-aR-a-R}{a^2-1}
  \end{array}
  \end{equation}

  At this moment, we have given a precise description of  $L(\tilde{V}_2)$; it is a strip delimited by the four vertices from above and the straight lines $y_{q_1}$, $y_{q_2}$, $y = \frac{x-(1-R)}{a}$ and $y = \frac{x-(1+R)}{a}$. In order to 
recover $\tilde{V}_2$ we should map it backward by $L^{-1}$. By continuity, and since $L$ is orientation preserving, 
the points which are leading to $y=R$ are $L^{-1}(\tilde{q}_4)$ and $L^{-1}(\tilde{q}_3)$ and the points that are mapped 
to $y=-R$ are $L^{-1}(\tilde{q}_2)$ and $L^{-1}(\tilde{q}_1)$, and therefore

  \begin{equation}\label{dist_tilde_V_2}
  d(\tilde{V}_2) = |L^{-1}(\tilde{q}_{4_x}) - L^{-1}(\tilde{q}_{3_x})| = |L^{-1}(\tilde{q}_{2_x}) - L^{-1}(\tilde{q}_{1_x})|
  \end{equation}
  
  \noindent
  For instance we compute $L^{-1}(\tilde{q}_{4_x})$ and $L^{-1}(\tilde{q}_{3_x})$

  \begin{equation}
  \begin{array}{ccl}
  L^{-1}(\tilde{q}_{4_x}) & = & L^{-1}(\frac{a^2\tilde{q}_{2_x}-aR-1-R}{a^2-1}, \frac{a\tilde{q}_{2_x}-aR-a-R}{a^2-1}) \\
                          &   &          \\
                          & = & (\frac{aR+a+R-a\tilde{q}_{2_x}}{a^2-1}, \frac{a^2\tilde{q}_{2_x}-aR-1-R}{a^2-1} + \frac{a|a\tilde{q}_{2_x}-aR-a-R|}{a^2-1}-1) \\
                          &   &          \\
                          & \stackrel{(*)}{=} & (\frac{aR+a+R-a\tilde{q}_{2_x}}{a^2-1}, \frac{a^2q_{2_x} - aR - 1 - R - a^2q_{2_x} + a^2R + a^2 + aR - a^2 + 1}{a^2-1}) \\
                          &   &          \\
                          & = & (\frac{aR+a+R-a\tilde{q}_{2_x}}{a^2-1},R) \\
  \end{array}
  \end{equation}
  
  \noindent
  In $(*)$ we have use the fact that $a\tilde{q}_{2_x}-aR-a-R < 0$ as $\tilde{q}_{2_x} < R$ and $a>0$. The computations 
for $L^{-1}(\tilde{q}_2)$ are analogous, so

  \begin{equation}
  L^{-1}(\tilde{q}_{3_x}) = (\frac{a+aR+R-aq_{1_x}}{a^2-1},R)
  \end{equation}
  Using \eqref{dist_tilde_V_2}, we can compute $d(\tilde{V}_2)$
  \begin{equation}
  \begin{array}{ccl}
  d(\tilde{V}_2) & = & \displaystyle{\frac{|a + aR + R - aq_{1_x} - a - aR - R + aq_{2_x}|}{a^2-1}} \\
                 &   &  \\
                 & = & \displaystyle{\frac{a}{a^2-1}\cdot |q_{2_x}-q_{1_x}| = \nu_v \cdot d(V)}
  \end{array}
  \end{equation}
  
  \noindent
 Therefore we only need the rate $\nu_v$ to be less than 1 and that holds when $a > \frac{1+\sqrt{5}}{2}$:
 
  \begin{equation}
  \frac{a}{a^2-1} < 1 \Leftrightarrow a^2 - a - 1 > 0 \Leftrightarrow a > \frac{1+\sqrt{5}}{2} = \Phi
  \end{equation}

  The last issue of assumption 2 that remains to be proved is that $\tilde{V}_i$ is a $\mu_v$-vertical strip. In order 
to prove this we realize that $L(\tilde{V}_i) = V \cap L(V_i) = V \cap H_i$ is a $\mu_h$-horizontal strip since its 
boundaries are $\mu_h$-horizontal curves. As we know by assumption 1, by $F^{-1}$ horizontal boundaries of horizontal 
strips map to horizontal boundaries of vertical strips and vertical boundaries of vertical strips map to vertical 
boundaries of horizontal strips and it is clear that this vertical boundaries of $\tilde{V}_i$ are $\mu_v$-horizontal 
curves. Therefore the second assumption is already proven and hence, using Theorem \ref{autonomous_main_theorem}, the 
existence of the chaotic saddle inside $S$ is proven.

\section{Non-autonomous Lozi map}
\label{sec:Lozinaut}

    In this section we prove that  Assumptions 1 and 3 of  the  nonautonomous Conley-Moser  conditions hold for the nonautonomous Lozi map. 
We begin by developing the set-up  for the problem. We define
  the maps $L_n$ and the domains $S_n$ as follows:

  \begin{equation}
  L_n(x,y) = (1 + y -a(n)|x|, \quad -x)
  \end{equation}
  
  \noindent
  and

  \begin{equation}
  L^{-1}_n(x,y) = (-y, \quad x + a(n)|y| - 1)
  \end{equation}

  \noindent
  where $a(n) = a + \epsilon(1 + \cos(n))$, $a > 4$. The domains can be set as $S_n = [-R(n),R(n)] \times 
  [-R(n),R(n)]$ where

  \begin{equation}
  R(n) = \frac{a(n)}{4(a(n)-2)}
  \end{equation}
  
  \noindent
  Although the domain may change with each iteration,  in this example we can consider a fixed domain in order to 
simplify the proof of Theorem \ref{non_autonomous_main_theorem}. 
  Therefore we define the domain as:

  \begin{equation}
  S = \sup_{n \in \mathbb{Z}}([-R(n), R(n)] \times [-R(n), R(n)]) = [-R, R] \times [-R, R]
  \end{equation}
  
  \noindent
  where

  \begin{equation}
  R = \sup_{n \in \mathbb{Z}} \frac{a(n)}{4(a(n)-2)}.
  \end{equation}

  \noindent
  To find this value we examine the monotonicity of the function
  
  \begin{equation}
  f(a(n)) = \frac{a(n)}{4(a(n)-2)}
  \end{equation}
  
  \noindent
  for $a(n) \neq 2$. This function always decreases since its derivate is negative when $a(n) \neq 2$. Therefore the  supreme of this value is determined from the following relations:

  \begin{equation}
  \inf_{n \in \mathbb{Z}}a(n) = a + \epsilon(1-1) = a.  
  \end{equation}
  
  \noindent
  and

  \begin{equation}
  R = \sup_{n \in \mathbb{Z}} R(n) = \frac{\inf_{n \in \mathbb{Z}}a(n)}{4(\inf_{n \in \mathbb{Z}}a(n)-2)} = 
  \frac{a}{4(a-2)}
  \end{equation}

  \noindent
  As we did in the autonomous case, we must check conditions similar to inequalities \eqref{first_condition} and 
  \eqref{second_condition} in order to obtain strips of four vertices totally included in the maximal domain $S$.
  \par
  \noindent
  By the symmetry of the problem, as it is shown in figure \ref{Lozi_cross}, the first condition, 
  \eqref{first_condition}, is satisfied since the point $p$ does not depend on the iteration $n$. So this condition holds when

  \begin{equation}
  R < \frac{1}{2}.
  \end{equation}
  
  \noindent
  The second condition is analogous to \eqref{second_condition}, and is also holds. We want the point $r$ to cut the 
  horizontal line $y=R$ out from the domain $S$ and it must have first coordinate lower than $-R$.

  \begin{equation}
  r = L_n(L_1^{-}) \cap \lbrace y = R \rbrace = (1 + R - a(n)|x|, -x)|_{-x=R} = (1 + R - a(n)R, R)
  \end{equation}
  
  \noindent
  therefore we must determine if the following inequality holds

  \begin{equation}
  1 + R - a(n)R < -R \Leftrightarrow 1 < (a(n)-2)R \quad \text{for all } n \in \mathbb{Z}
  \end{equation}
  
  \noindent
  and it does hold  since

  \begin{equation}
  (a(n)-2)R > (a -2)\frac{a}{4(a-2)} = \frac{a}{4} > 1
  \end{equation}
  
  \noindent
  when $a > 4$ for all $n \in \mathbb{Z}$.
  \par
  \noindent
  Henceforth, we set $S = S_n$ for all $n \in \mathbb{Z}$ and we can state the main nonautonomous theorem.

  \begin{theorem}\label{non_autonomous_main_theorem}
  For $a(n) = a + \epsilon(1 + \cos(n))$ where $a>4$ and $\epsilon$ small, Assumptions 1 and 3  of the nonautonomous Conley-Moser conditions hold for the nonautonomous Lozi map $L_n$. and therefore  it possesses a chaotic saddle, $\Lambda_n$, inside the square $S$. In other words, $\left\lbrace 
  \Lambda_n, L_n \right\rbrace ^{+\infty}_{n=-\infty}$ is a infinite sequence of chaotic invariant sets where 
  $$\Lambda \equiv \cup_{n \in \mathbb{Z}} \Lambda_n,$$ contains an uncountable infinity or orbits, each of them 
  unstable (of saddle type) and the dynamics on the invariant set exhibits sensitive dependence on initial 
  conditions. 
  \end{theorem}

  From now to on, we see how these two assumptions are held. First of all, we construct the horizontal and vertical 
  strips as in the autonomous case. Then, we check the third assumption, the \textit{cone condition}, to quantify the
 folding and stretching of the strips. 

\subsection{Verification of Assumption 1 for the nonautonomous Conley-Moser Conditions}
\label{sec:nonautCM1}

We have defined the bi-infinite sequence of maps and domains
  
  \begin{equation}
  \left \{L_n,S\right \}^{+\infty}_{n=-\infty}, \qquad L_n:S \rightarrow \mathbb{R}^2
  \end{equation}
  
  \noindent
  satisfying
  
  $$L_n(S) \cap S \neq \emptyset, \quad \forall n\in \mathbb{Z}.$$
  
  \noindent
Now we define the main geometrical structures of this problem, the horizontal and the vertical strips. In this context we will proceed as in the autonomous case adding the iteration variable

  \begin{equation}
  \begin{array}{ccc}
  H^{n+1}_1 \cup H^{n+1}_2 & := & L_n(S) \cap S       \\
                           &    &                     \\
  V^{n}_1 \cup V^{n}_2     & := & L^{-1}_n(S) \cap S  \\
  \end{array}
  \end{equation}
  
  \noindent
  where $H^{n+1}_1$ is the upper half part of $L_n(S) \cap S$ and $H^{n+1}_2$ is the lower half part of $L_n(S) \cap S$, 
  $V^n_1$ is the left half part of $L^{-1}_n(S) \cap S$ and $V^n_2$ is the right half part of $L^{-1}_n(S) \cap S$. We 
  will use the following notation

  \begin{equation}
  \begin{array}{ccc}
  H^{n+1}_1 & := & (L_n(S) \cap S)^{+}\text{, where } y > 0, \\
            &    &                   \\
  H^{n+1}_2 & := & (L_n(S) \cap S)^{-}\text{, where } y < 0, \\
            &    &                             \\
  V^{n}_1   & := & (L^{-1}_n(S) \cap S)^{-}\text{, where } x < 0, \\
            &    &                   \\
  V^{n}_2   & := & (L^{-1}_n(S) \cap S)^{+}\text{, where } x > 0. \\
  \end{array}
  \end{equation}  
  
  We are giving only the vertices of $H^{n+1}_1$ and $V^n_1$. The vertices of $H^{n+1}_2$ and $V^n_2$ are symmetric 
  with respect to $y=0$ and $x=0$ respectively. We show $H^{n+1}_i$ and $V^n_i$ in figure \ref{No_auto_strips}.   
  
  \begin{equation}
  h_{11} : \lbrace x = -R \rbrace \cap \left \{y = \frac{1+R-x}{a(n)} \right \} \Rightarrow \left \{
                                                            \begin{array}{l}
                                                            h_{11_x} = -R \\
                                                                          \\
                                                            h_{11_y} = \frac{1+2R}{a(n)}  \\
                                                            \end{array}
                                                            \right .
  \end{equation}

  \begin{equation}
  h_{12} : \lbrace x = R \rbrace \cap \left \{y = \frac{1+R-x}{a(n)} \right \} \Rightarrow \left \{
                                                            \begin{array}{l}
                                                            h_{12_x} = R \\
                                                                          \\
                                                            h_{12_y} = \frac{1}{a(n)}  \\
                                                            \end{array}
                                                            \right .
  \end{equation}

  \begin{equation}
  h_{13} : \lbrace x = R \rbrace \cap \left \{y = \frac{1-R-x}{a(n)} \right \} \Rightarrow \left \{
                                                            \begin{array}{l}
                                                            h_{13_x} = R \\
                                                                          \\
                                                            h_{13_y} = \frac{1-2R}{a(n)}  \\
                                                            \end{array}
                                                            \right .
  \end{equation}

  \begin{equation}
  h_{14} : \lbrace x = -R \rbrace \cap \left \{y = \frac{1-R-x}{a(n)} \right \} \Rightarrow \left \{
                                                            \begin{array}{l}
                                                            h_{14_x} = -R \\
                                                                          \\
                                                            h_{14_y} = \frac{1}{a(n)}  \\
                                                            \end{array}
                                                            \right .
  \end{equation}
  
  \noindent
  and the computations for the vertices of $V^n_1$ are similar

  \begin{equation}
  v_{11} : \lbrace y = R \rbrace \cap \left \{y = -R-a(n)x-1  \right \} \Rightarrow \left \{
                                                        \begin{array}{l}
                                                        v_{11_x} = \frac{-2R-1}{a(n)} \\
                                                                           \\
                                                        v_{11_y} = R  \\
                                                        \end{array}
                                                        \right .
  \end{equation}

  \begin{equation}
  v_{12} : \lbrace y = R \rbrace \cap \left \{y = R-a(n)x-1 \right \} \Rightarrow \left \{
                                                        \begin{array}{l}
                                                        v_{12_x} = \frac{-1}{a(n)} \\
                                                                           \\
                                                        v_{12_y} = R  \\
                                                        \end{array}
                                                        \right .
  \end{equation}

  \begin{equation}  
  v_{13} : \lbrace y = -R \rbrace \cap \left \{y = R-a(n)x-1 \right \} \Rightarrow \left \{
                                                        \begin{array}{l}
                                                        v_{13_x} = \frac{2R-1}{a(n)} \\
                                                                           \\
                                                        v_{13_y} = -R  \\
                                                        \end{array}
                                                        \right .
  \end{equation}

  \begin{equation}
  v_{14} : \lbrace y = -R \rbrace \cap \left \{y = -R-a(n)x-1 \right \} \Rightarrow \left \{
                                                        \begin{array}{l}
                                                        v_{14_x} = \frac{-1}{a(n)} \\
                                                                           \\
                                                        v_{14_y} = -R  \\
                                                        \end{array}
                                                        \right .
  \end{equation}
  
\begin{figure}[htbp!]
\centering
\subfigure[$\quad n_0 = n$]{\includegraphics[width=0.45\linewidth]{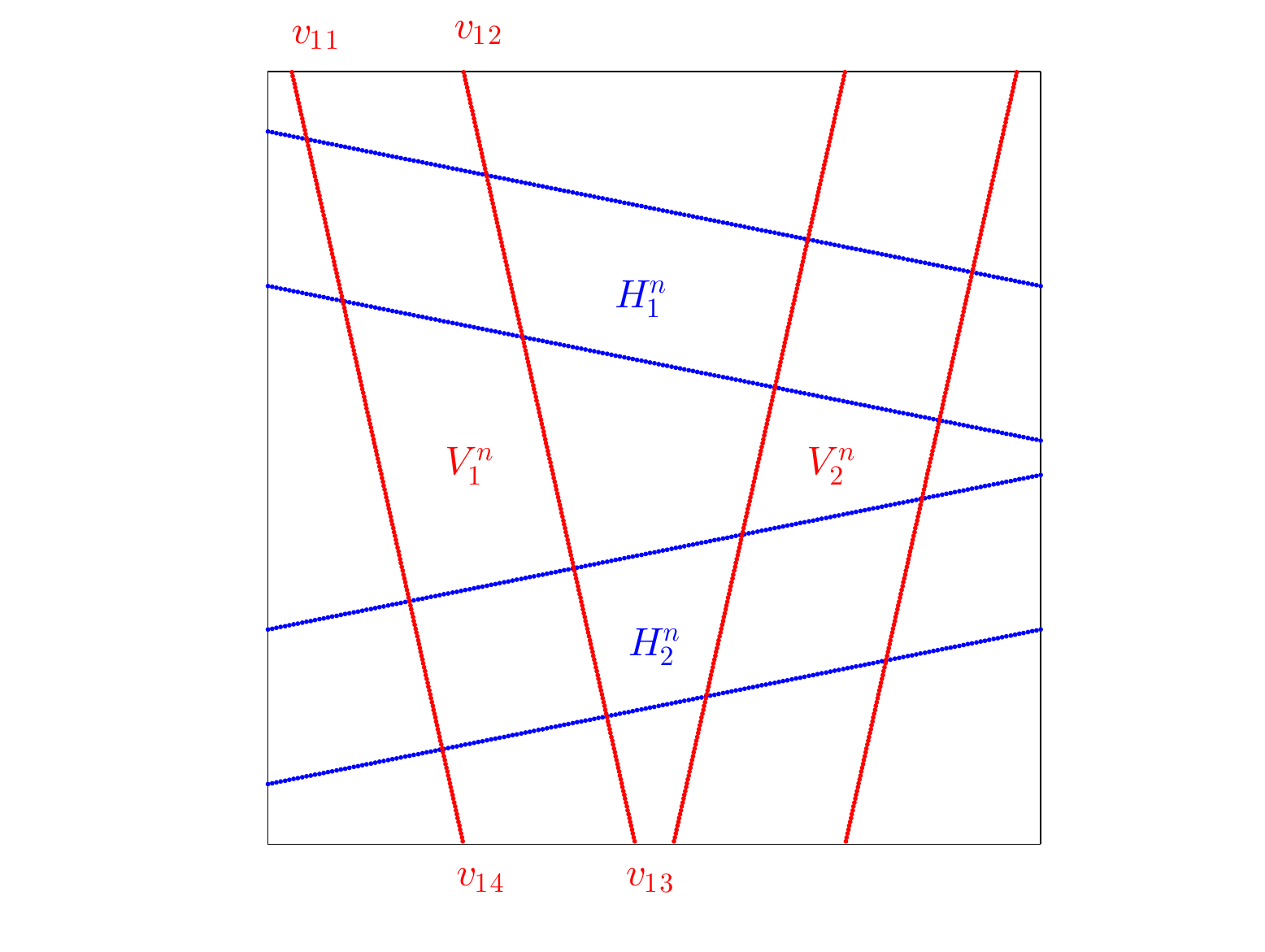}} 
\subfigure[$\quad n_0 = n+1$]{\includegraphics[width=0.45\linewidth]{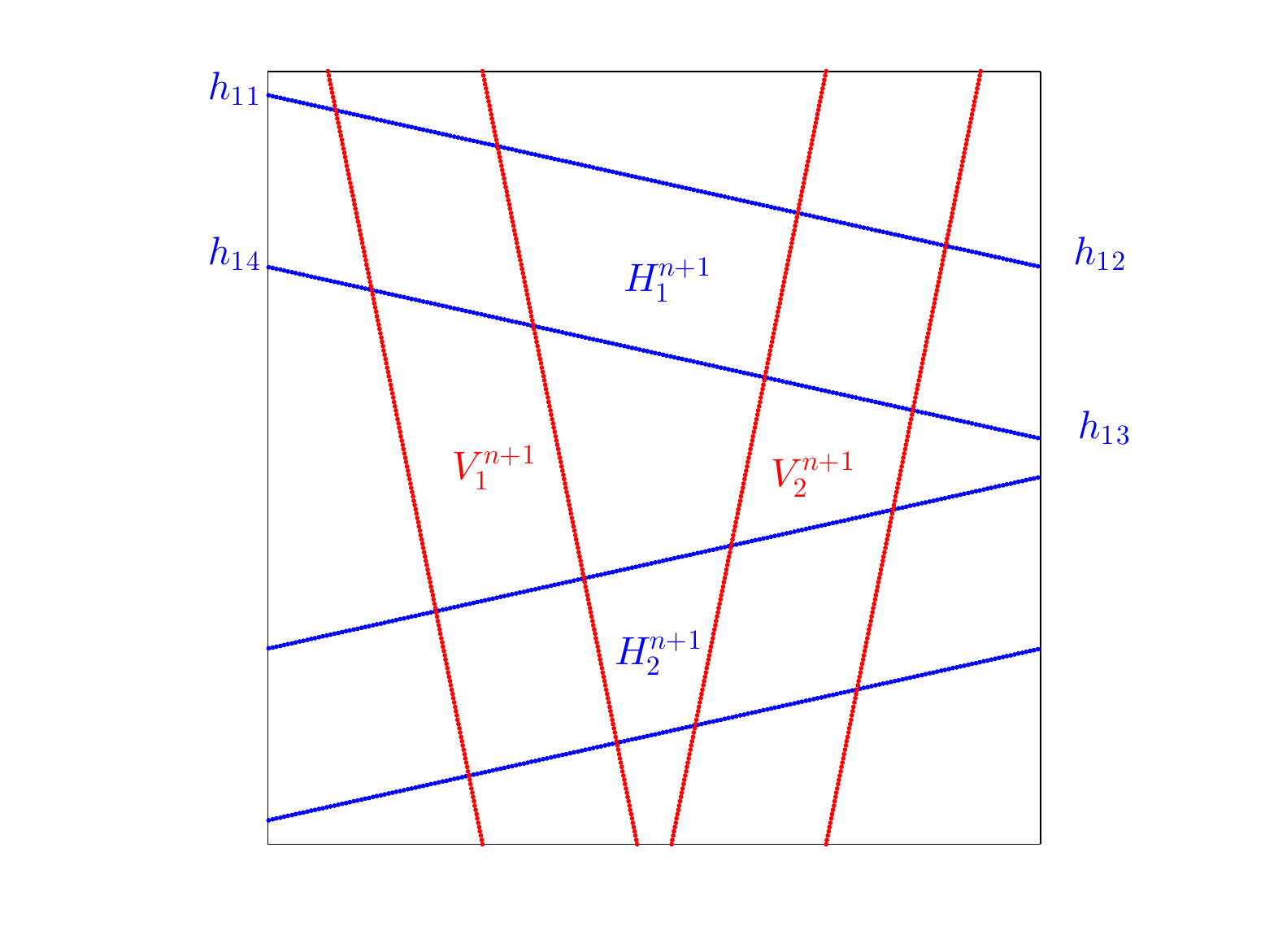}} \\
\caption{Vertical and horizontal strips. $L_n(V^n_i) = H^{n+1}_i, \quad i=1,2$.}
\label{No_auto_strips}
\end{figure}

  As it is defined above, $V^{n}_i$ and $H^{n+1}_i$ are $\mu^n_v$-vertical strips and $\mu^{n+1}_h$-horizontal strips 
respectively with $\mu^n_v = 1/a(n)$ and $\mu^{n+1}_h = 1/a(n+1)$. Taking into account Remark \ref{slope_remark} and 
due to reasons explained afterwards, by convenience, we can choose
\begin{equation}
  \mu_v = \mu_h = \frac{a-\sqrt{a^2-4}}{2} > 1/a(n),
\end{equation}
  for all $n \in \mathbb{Z}$ and $a>4$.

  \begin{remark}\label{slope_remark}
  Note that if $V$ is a $\mu_v$-vertical strip and $\mu_v \leq \mu^{*}_v$, therefore $V$ is a $\mu^{*}_v$-vertical 
strip:
  $$|x_1-x_2| \leq \mu_v |y_1-y_2| \leq \mu^{*}_v |y_1-y_2|.$$
  The same argument could be done for horizontal strips.  
  \end{remark}
  \par
  \noindent
  $L_n$ is a homeomorphism on the whole plane and so on $S$. Therefore

  \begin{equation}
  L_n(V^n_i) = H^{n+1}_i, \quad i=1,2.
  \end{equation}
  
  \noindent
  We should also define

  \begin{equation}
  \begin{array}{rcl}\label{second_strips}
  L_n(V^{n}_i) \cap V^{n+1}_j    & \equiv & H^{n+1}_{ij} \\
                                 &        &              \\
  V^n_i \cap L^{-1}_n (V^{n+1}_j)& \equiv & L^{-1}_n(H^{n+1}_{ij}) \equiv V^n_{ji} \\
  \end{array}
  \end{equation}

  Let $\left \{A^n\right \}^{+\infty}_{n=-\infty}$ denote a sequence of $2 \times 2$ matrices such that

  \begin{equation}
  A^n_{ij} = \left \{\begin{array}{ccc}
                     0 & \text{if} & L_n(V^n_i) \cap V^{n+1}_j = H^{n+1}_i \cap V^{n+1}_j = \emptyset    \\
                       &           &                                                                     \\
                     1 & \text{if} & L_n(V^n_i) \cap V^{n+1}_j = H^{n+1}_i \cap V^{n+1}_j \neq \emptyset \\
                     \end{array}\right.
  \end{equation}
  
  \noindent
  In our case, we can ensure that
  \begin{equation}
  A^{n} = \left ( \begin{array}{cc}
               1 & 1 \\
               1 & 1 \\
               \end{array}
      \right )
  \end{equation}

  We know that $H^{n+1}_i$ are horizontal strips formed by two $\mu_h$-horizontal curves that rest on the lines $x=R$ 
  and $x=-R$ and $V^{n+1}_j$ are vertical strips formed by two $\mu_v$-vertical curves that lead from $y=R$ to $y=-R$. 
  So applying the \textit{Fixed Point Theorem}, these straight lines get cut in four vertices that together with the 
  two horizontal and the two vertical lines form $H^{n+1}_{ij}$. These sets are $\mu_h$-horizontal strips because its 
  boundaries  are contained in $V^{n+1}_j$ due to its construction ($H^{n+1}_i \cap V^{n+1}_j$).
  \par
  \noindent
  Furthermore, due to the fact that $L_n$ is homeomorphism and the construction of the strips, $L_n$ maps $V^n_{ji}$ 
  onto $H^{n+1}_{ij}$. It is clear from \eqref{second_strips} that

  \begin{equation}
  L^{-1}_n(H^{n+1}_{ij}) = V^n_{ji} = L^{-1}_n(V^{n+1}_j) \cap V^n_i \Longrightarrow L^{-1}_n(H^{n+1}_{ij}) 
  \subset V^n_i
  \end{equation}  
  
  \noindent
  so

  \begin{equation}
  L_{n}^{-1}(\partial_{h}H_{ij}^{n+1}) \subset \partial_{h}V_{i}^{n}.
  \end{equation}
  
  \noindent
  Finally we need to prove $\mu_v\mu_h < 1$. We know that $\mu_v = \mu_h = \frac{a-\sqrt{a^2-4}}{2}$ and
  
  \begin{equation}
  \mu_v \cdot \mu_h = \displaystyle{\left (\frac{a-\sqrt{a^2-4}}{2}\right )^2} < 1
  \end{equation}
  so Assumption 1 is  proven.

\subsection{The Cone Condition. Assumption A3 of the Nonautonomous Conley-Moser Conditions}

  As we stated earlier, the second assumption of the Conley-Moser conditions can be replaced by another condition named the 
\textit{cone condition}. Given any point $z_0 = (x_0,y_0) \in {\cal H}^{n+1}$ and any $(\xi_{z_0},\eta_{z_0}) \in 
S^s_{{\cal H}^{n+1}}$ (which by definition, $|\xi_{z_0}| \leq \mu_v|\eta_{z_0}|$) we have that

\begin{equation}\label{inverseJacob}
  DL^{-1}_n(z_0)\cdot (\xi_{z_0},\eta_{z_0}) = \left ( \begin{array}{cc}
					                   0   & -1          \\
							   1   & a(n)sign(y) \\
						       \end{array}\right ) \left ( \begin{array}{c}
									             \xi_{z_0}  \\
									             \eta_{z_0} \\
									           \end{array}\right ) = 
									           \left ( \begin{array}{c}
									             -\eta_{z_0}  \\
									            \xi_{z_0} + a(n)sign(y)\eta_{z_0} \\
									           \end{array}\right )
\end{equation}

\noindent
and \eqref{inverseJacob} belongs to $S^s_{{\cal V}^n}$ if and only if the inequality

\begin{equation}
  |-\eta_{z_0}| = |\eta_{z_0}| \leq \mu_v|\xi_{z_0} + a(n)sign(y)\eta_{z_0}|
\end{equation}

\noindent
holds and this is true since

\begin{equation}
\begin{array}{rcl}
  \mu_v|\xi_{z_0} + a(n)sign(y)\eta_{z_0}| \geq \mu_v (a(n)|\eta_{z_0}| - |\xi_{z_0}|) & \geq & \\
                                                                                       &      & \\
  \mu_v(a(n)|\eta_{z_0}| - \mu_v|\eta_{z_0}|)= \mu_v(a(n) - \mu_v)|\eta_{z_0}| & \stackrel{(*)}{\geq} & |\eta_{z_0}| \\
\end{array}
\end{equation}

\noindent
Last part of the inequality, (*), holds when $\mu_v (a(n) - \mu_v) \geq 1$ and that is satisfied when

\begin{equation}
  \mu_v \in \left [\frac{a(n)-\sqrt{a(n)^2-4}}{2}, \quad \frac{a(n)+\sqrt{a(n)^2-4}}{2} \right ] = I_n
\end{equation}

\noindent
We need to prove that $\mu_v \in I_n$ for every $n \in \mathbb{Z}$, or, equivalently

\begin{equation}
\begin{array}{rcl}
  \mu_v \in \bigcap_{n \in \mathbb{Z}} I_n & = & \left [\sup_{n \in \mathbb{Z}} \frac{a(n)-\sqrt{a(n)^2-4}}{2}, \quad 
\inf_{n \in \mathbb{Z}} \frac{a(n)+\sqrt{a(n)^2-4}}{2} \right ] \\
					   &   &  \\
					   & = &  \left [ \displaystyle{\frac{a-\sqrt{a^2-4}}{2}}, \quad 
\displaystyle{\frac{a+\sqrt{a^2-4}}{2}} \right ] \\
\end{array}
\end{equation}
Finally we can set a general value for $\mu_h$ and $\mu_v$ that hold  for the last inequalities

\begin{equation}
  \mu_h = \mu_v = \frac{a-\sqrt{a^2-4}}{2} 
\end{equation}

\noindent
and we can observe that

\begin{equation}
  \mu_h \cdot \mu_v = \displaystyle{\left (\frac{a-\sqrt{a^2-4}}{2}\right )^2}<1
\end{equation}

\noindent
Since $z_0 \in {\cal H}^{n+1}$ is an arbitrary point, the inclusion $DL^{-1}_n(S^s_{{\cal H}^{n+1}}) \subset S^s_{{\cal 
V}^n}$ is  proven.
\par
\noindent
To finish the proof of Assumption 3, we only need to  prove the inequality

\begin{equation}
  |\eta_{{f}^{-1}_n(z_0)}| \geq \frac{1}{\mu}|\eta_{z_0}|
\end{equation}

\noindent
for $0<\mu<1-\mu_h\mu_v$ and $z_0 \in {\cal H}^{n+1}$, $(\xi_{z_0},\eta_{z_0}) \in S^s_{{\cal H}^{n+1}}$ because
\begin{equation}
  |\xi_{{f}_n(z_0)}| \geq \frac{1}{\mu}|\xi_{z_0}|, \quad z_0 \in {\cal V}^n, \quad (\xi_{z_0},\eta_{z_0}) \in 
S^u_{{\cal V}^n}
\end{equation}

\noindent
is proved by using a similar argument.

\begin{equation}
\begin{array}{rcl}
  |\eta_{{f}^{-1}_n(z_0)}| \geq |\xi_{z_0} + a(n)sign(y)\eta_{z_0}| \geq a(n)|\eta_{z_0}| - |\xi_{z_0}| & \geq & \\
                                                                                                        &      & \\
  a(n)|\eta_{z_0}| - \mu_v|\eta_{z_0}| = (a(n) - \mu_v)|\eta_{z_0}| & \geq & \frac{1}{\mu}|\eta_{z_0}|\\
\end{array}
\end{equation}

\noindent
so it follows that
\begin{equation}
  \mu \geq \frac{1}{a(n)-\mu_v}.
\end{equation}

\noindent
Taking into account this last inequality and the condition $0<\mu<1-\mu_h\mu_v$, we have the following chain of 
inequalities

\begin{equation}
  \frac{1}{a(n)-\mu_v} \leq \mu \leq 1-\mu_h\mu_v
\end{equation}

\noindent
and this last inequality is satisfied providing that $a>4$ so the proof of the Assumption 3 is complete.
\par
\noindent
The proof of the second inclusion $DL_n(S^u_{{\cal V}^n}) \subset S^u_{{\cal H}^{n+1}}$ is similar.

\section{The visualization of the Chaotic Saddle.}
\label{sec:DLD}

  In this section the chaotic saddle for the autonomous and for the non-autonomous cases are displayed by using the discrete Lagrange descriptors (DLD). This tool explained in \cite{Lopesino} consists of evaluating the $p$-norm of an orbit generated by a two dimensional map, in our case the Lozi map. For instance, let
\begin{equation}
\left\lbrace x_n,y_n \right\rbrace^{n=N}_{n=-N}, \quad N \in \mathbb{N},
\end{equation}
denote an orbit of a point. The DLD is defined as follows:
\begin{equation}\label{DLD}
MD_p = \displaystyle{\sum^{N-1}_{i=-N} |x_{i+1}-x_i|^p + |y_{i+1}-y_i|^p}, \quad p \leq 1.
\end{equation}
We are interested on showing the phase space where the chaotic saddle exists, therefore our study region is the square 
$S$. We prepare a set of  initial conditions by fixing a spatial grid and then expression \eqref{DLD} is applied to  the 
initial conditions belonging to this grid. 

{\bf The Chaotic Saddle for the autonomous case.} Figure \ref{Chaotic_Saddle} shows the chaotic saddle of the Lozi map 
for different values of $a$. It is confirmed that only when $a \geq 4$, the set is wholly contained in the square S.

\begin{figure*}[htbp!]
\centering
\subfigure[$\quad$ Chaotic S with a=3]{\includegraphics[width=0.45\linewidth]{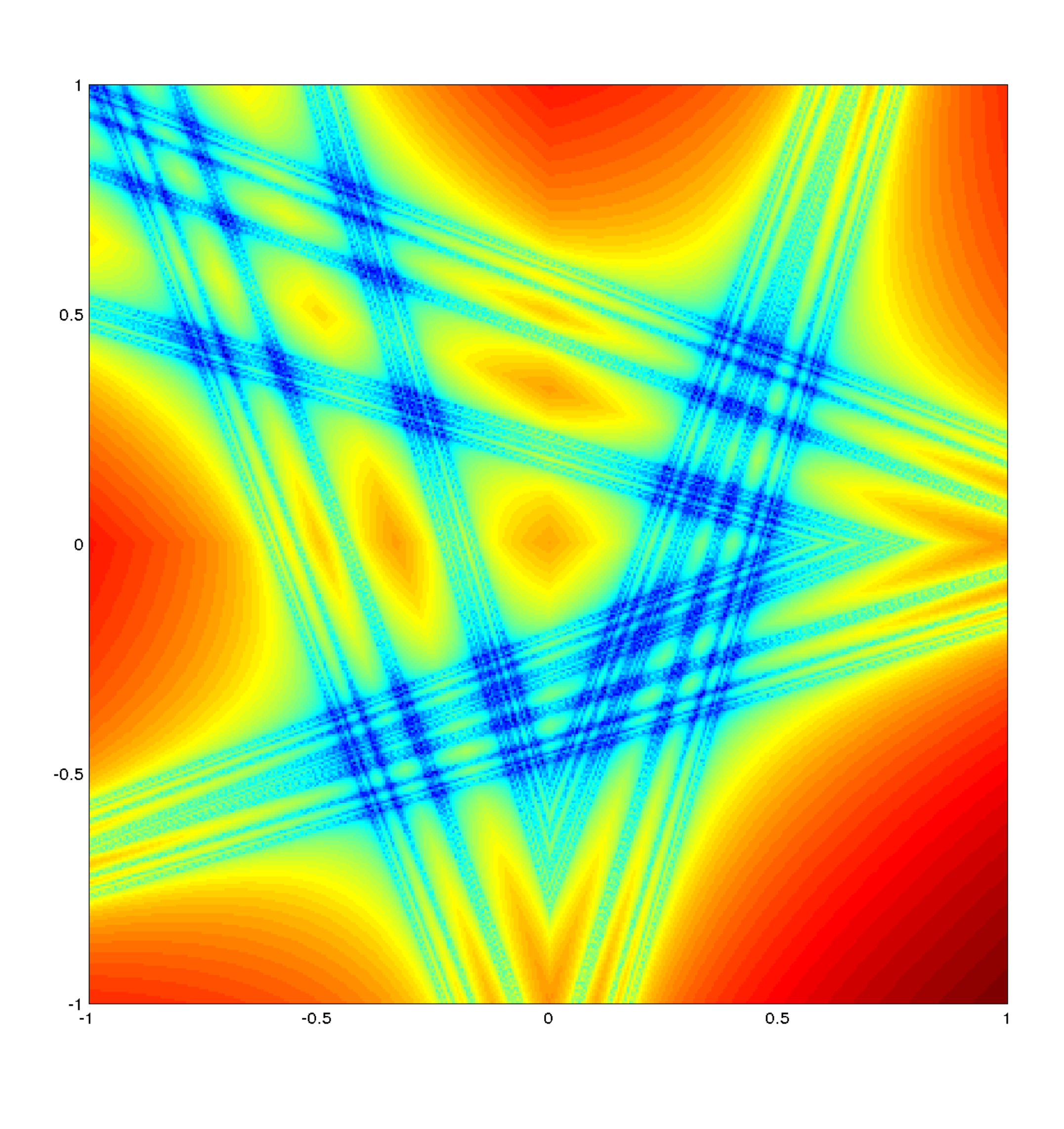}} 
\subfigure[$\quad$ Chaotic S with a=3.5]{\includegraphics[width=0.45\linewidth]{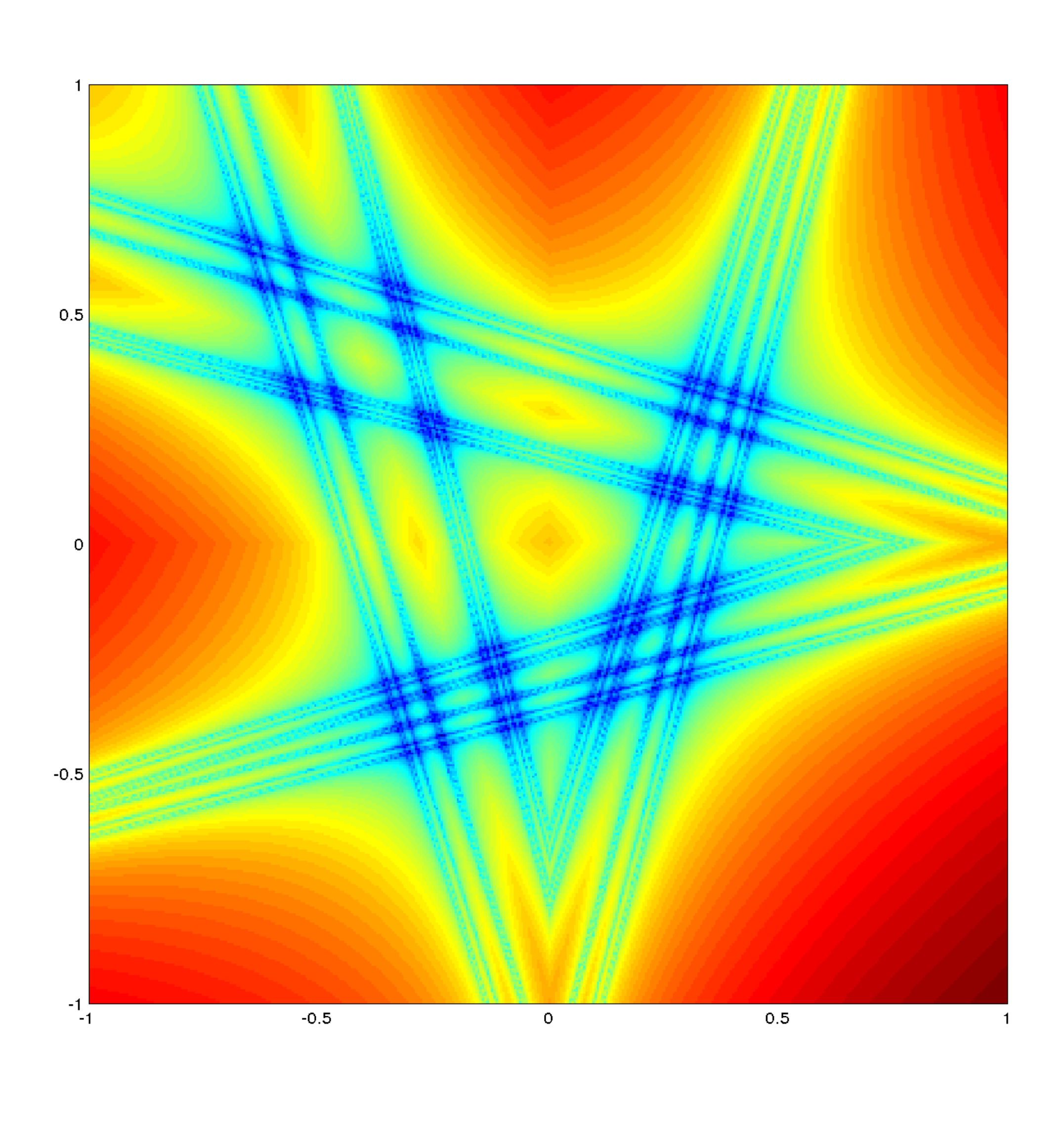}} \\
\subfigure[$\quad$ Chaotic S with a=4]{\includegraphics[width=0.45\linewidth]{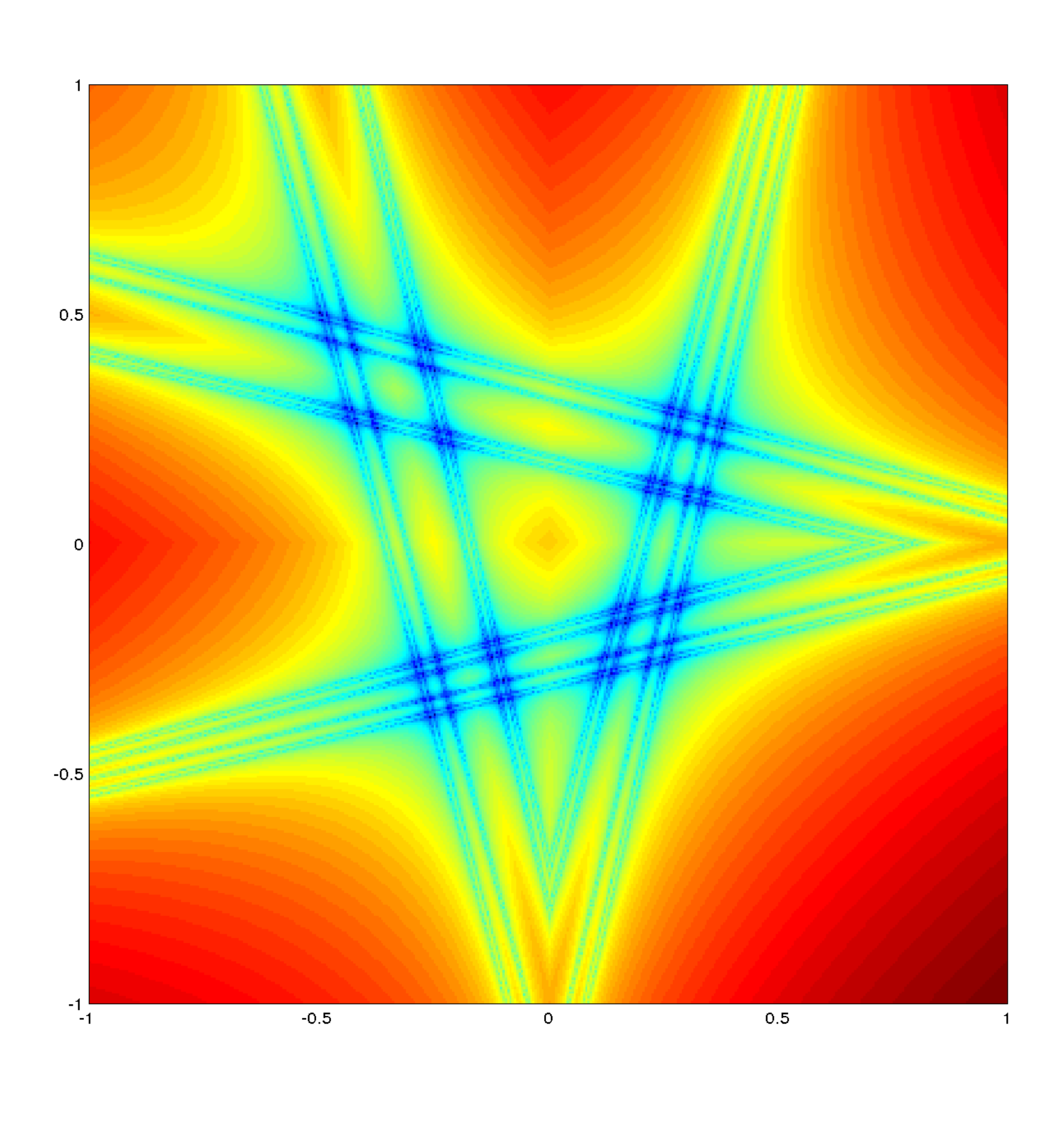}}
\subfigure[$\quad$ Chaotic S with 
a=4.5]{\includegraphics[width=0.45\linewidth]{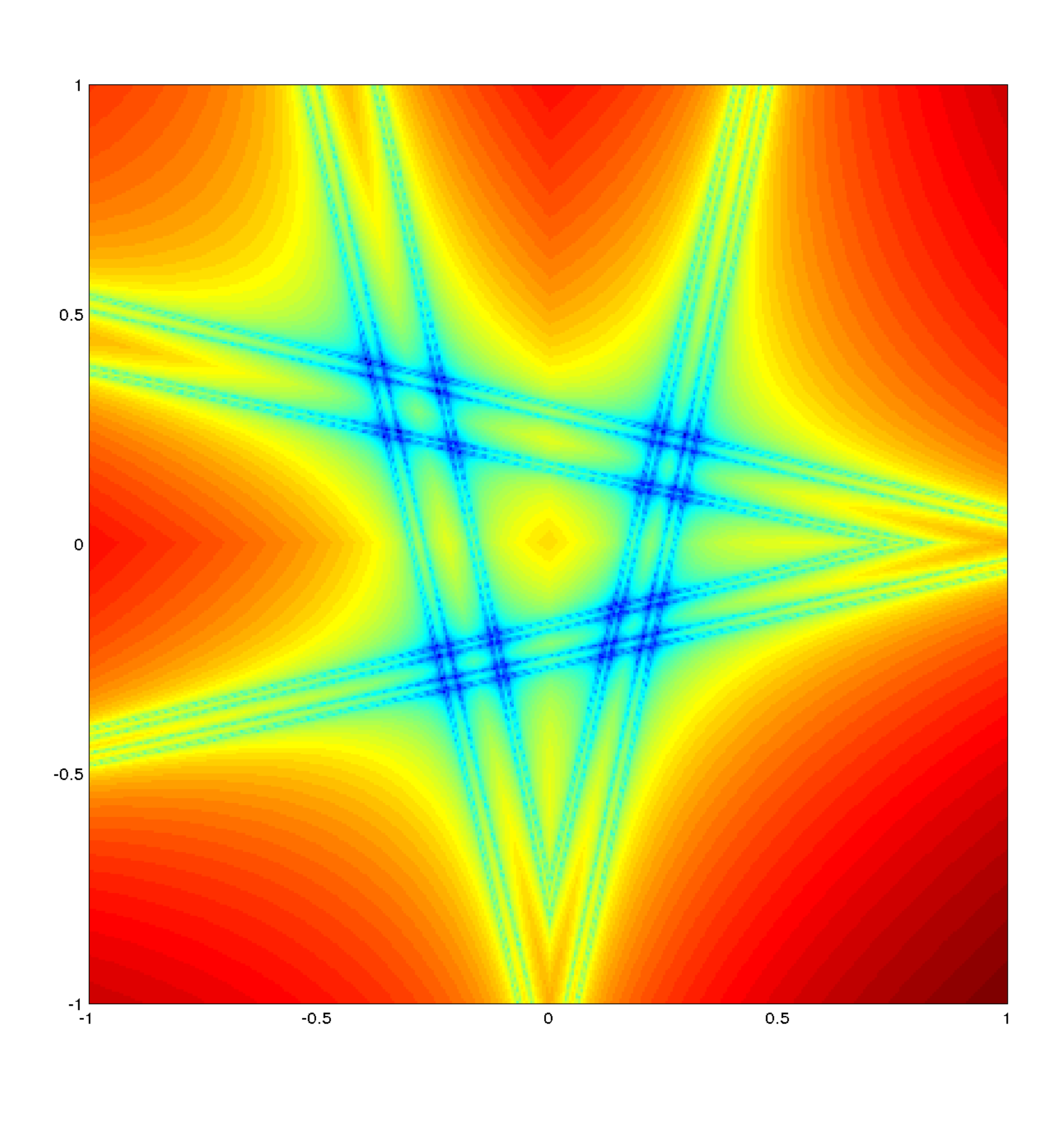}}\qquad\hfill\caption{Chaotic saddle for different 
values of $a$. These panels show contours of $MD_p$ for $p=0.25$ and $N=20$, 
with a grid point spacing of $0.005$.}
\label{Chaotic_Saddle}
\end{figure*}

{\bf The Chaotic Saddle for the non autonomous case.} Figure \ref{Chaotic_Saddle_non_autonomous} shows by means of the DLD tool, how the chaotic saddle evolves with the iterations when the autonomous system is perturbed. 

\begin{figure*}[htbp!]
\centering
\subfigure[$\quad$ Chaotic S. with $n_0=-3$]{\includegraphics[width=0.45\linewidth]{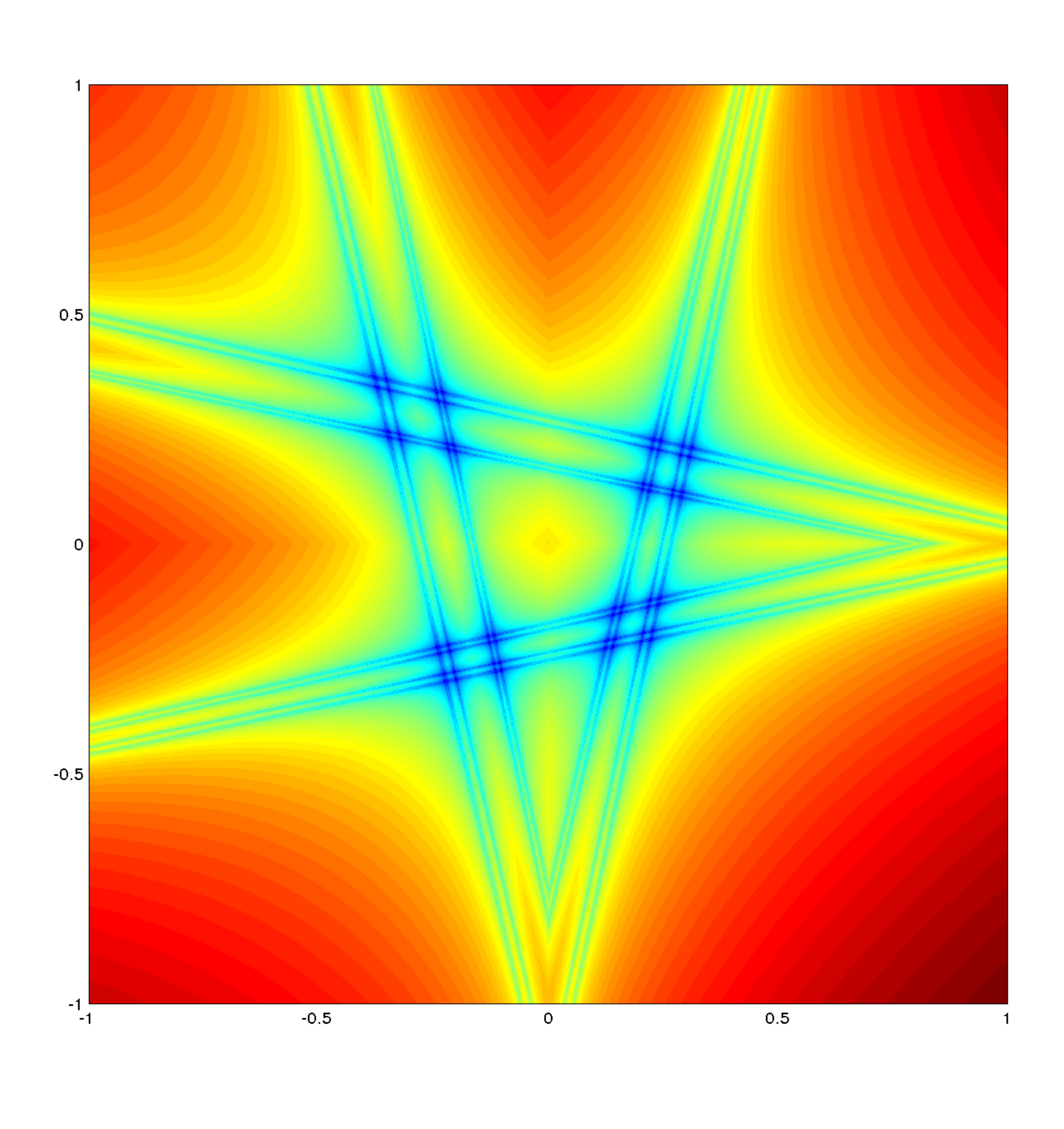}} 
\subfigure[$\quad$ Chaotic S. with $n_0=-1$]{\includegraphics[width=0.45\linewidth]{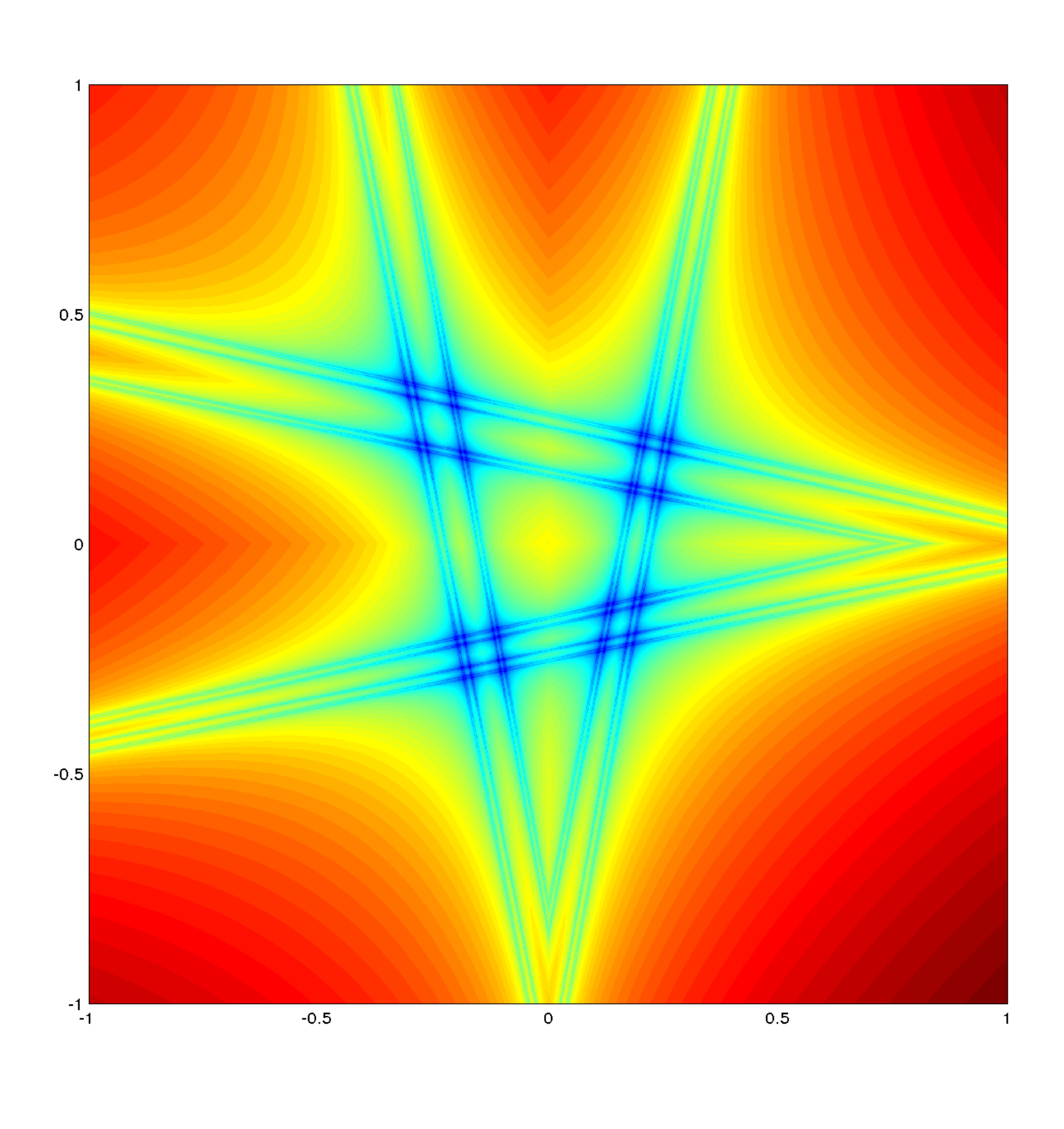}} \\
\subfigure[$\quad$ Chaotic S. with $n_0=1$]{\includegraphics[width=0.45\linewidth]{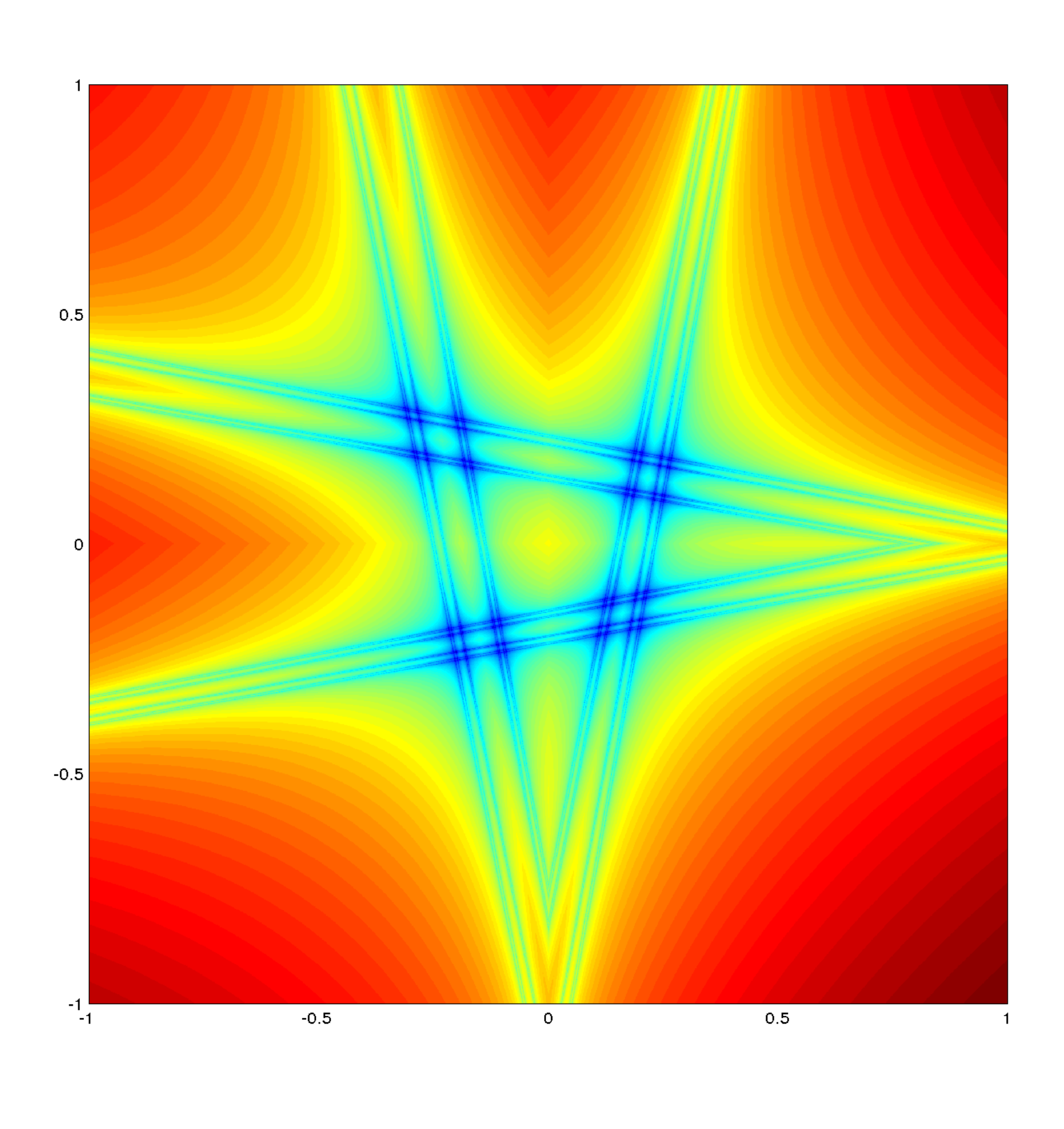}}
\subfigure[$\quad$ Chaotic S. with 
$n_0=3$]{\includegraphics[width=0.45\linewidth]{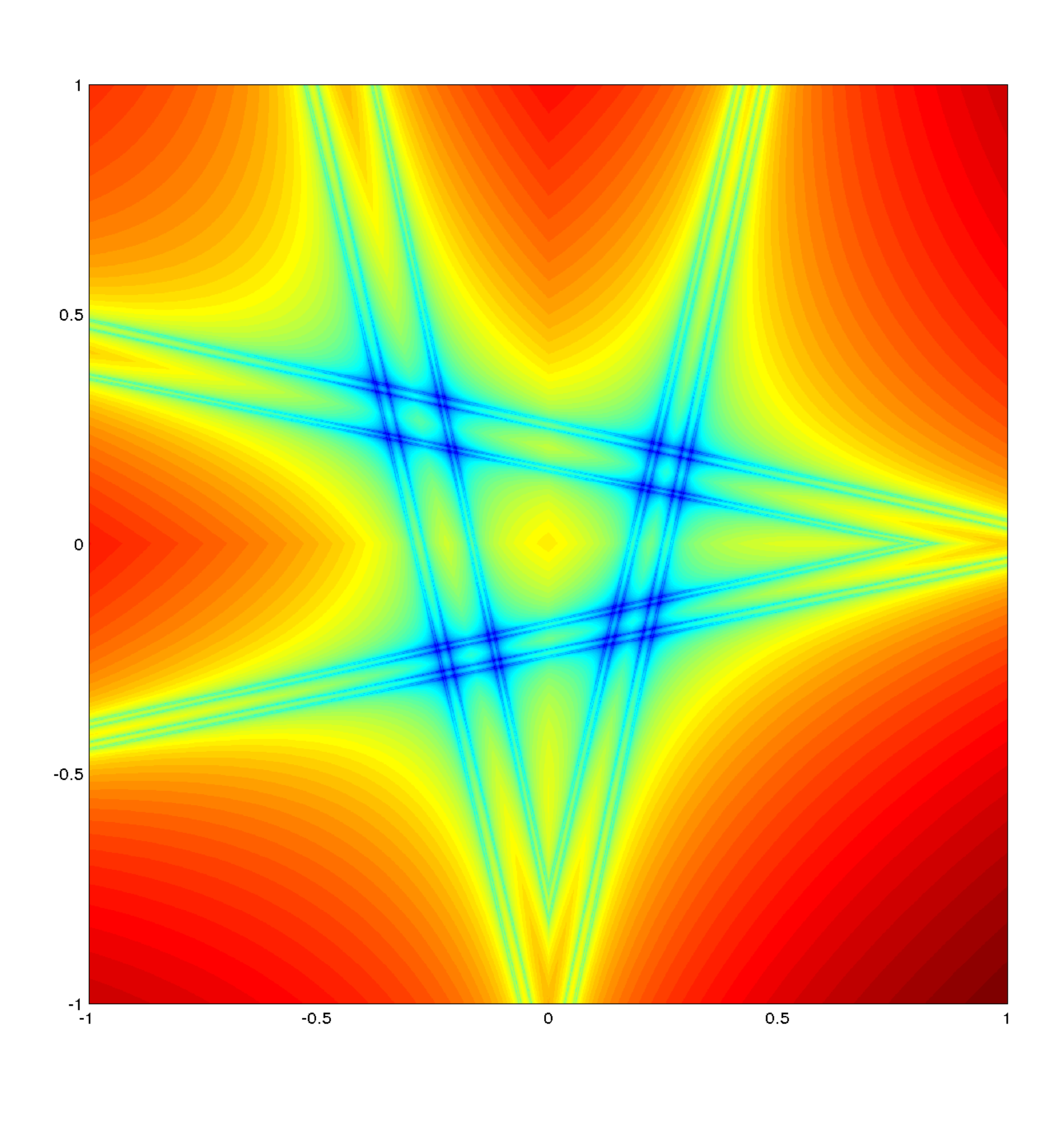}}\qquad\hfill\caption{Chaotic 
saddle for different starting time iteration. These panels show contours of $MD_p$ for $p=0.25$, $A=4.5 + 
\epsilon(1+\cos(n))$ and $N=100$, with a grid point spacing of $0.001$.}
\label{Chaotic_Saddle_non_autonomous}
\end{figure*}

\section{Summary and Conclusions}
\label{sec:summ}

In this paper we have considered the Lozi map, both in its autonomous and nonautonomous versions, and provided necessary conditions for the map to possess a  chaotic invariant set. This is accomplished by using autonomous and nonautonomous  versions of the Conley-Moser conditions, in particular we used the sharpened  conditions  for nonautonomous maps given in \cite{balibrea} to show that the nonautonomous chaotic invariant set is hyperbolic. In the course of the proof we provide a precise characterization of what is meant by the phrase ''hyperbolic chaotic invariant set'' for nonautonmous dynamical systems. At the end of this paper we have used the DLD to visualize the chaotic saddle for different parameters.

\section*{\bf Acknowledgments.} The research of CL, FB-I and AMM is supported by the MINECO under grant MTM2014-56392-R. The 
research of SW is supported by  ONR Grant No.~N00014-01-1-0769.  We acknowledge support from MINECO: ICMAT Severo Ochoa 
project SEV-2011-0087.


\begin{thebibliography}{}

\bibitem[Alekseev(1968a) Alekseev]{vma_a}
Alekseev, V.M. (1968a).
\newblock Quasirandom dynamical systems, {I}.
\newblock { Math. USSR-Sb }, {\bf 5}, 73-128.

\bibitem[Alekseev(1968b) Alekseev]{vma_b}
Alekseev, V.M. (1968b).
\newblock Quasirandom dynamical systems, {II}.
\newblock { Math. USSR-Sb }, {\bf 6}, 505-560.

\bibitem[Alekseev(1969)Alekseev]{vma_c}
Alekseev, V.M. (1969).
\newblock Quasirandom dynamical systems, {III}.
\newblock { Math. USSR-Sb }, {\bf 6}, 1-43.

\bibitem[Balibrea-Iniesta, {\em et~al.}(2015)Balibrea]{balibrea}
Balibrea-Iniesta, F., Lopesino, C.,  Wiggins S., Mancho A.M. (2015).
\newblock Chaotic dynamics in nonautonomous maps: Application to the nonautonomous H\'enon map.
\newblock {\em International Journal of Bifurcation and Chaos} (in press).

\bibitem[Chastaing, {\em et~al.}(2014)]{ChasBerGem}
Chastaing J. -Y., Bertin E. and G\'eminard J. -C. (2014).
\newblock Dynamics of the bouncing ball.
\newblock http://arxiv.org/pdf/1405.3482v1.pdf

\bibitem[Devaney and Nitecki(1979)]{Dev79}
Devaney, R. and Nitecki, Z. (1979).
\newblock Shift automorphisms in the h\'enon mapping.
\newblock {\em Comm. Math. Phys.}, {\bf 67}, 137--179.

\bibitem[Holmes(1982)]{Holmes}
Holmes P. J. (1982).
\newblock The dynamics of repeated impacts with a sinusoidally vibrating table.
\newblock {Journal of Sound and Vibration}, {\bf 84(2)}, 173-189.

\bibitem[Lerman and Silnikov(1992)Lerman and Silnikov]{ls}
Lerman, L. and Silnikov, L. (1992).
\newblock Homoclinical structures in nonautonomous systems: Nonautonomous chaos.
\newblock {\em Chaos\/}, {\bf 2}, 447--454.

\bibitem[Lopesino, {\em et~al.}(2015)Lopesino]{Lopesino}
Lopesino C., Balibrea F., Wiggins S., Mancho A. (2015).
\newblock Lagrangian Descriptors for Two Dimensional, Area Preserving, Autonomous and Nonautonomous Maps.
\newblock {\em Communications in Nonlinear Science and Numerical Simulation} {\bf 27} (1-3)  40--51.

\bibitem[Lozi(1978)Lozi]{Lozi}
Lozi, R. (1978).
\newblock Un attracteur etrange (?) du type attracteur de Henon.
\newblock {Journal de Physique Colloques}, {\bf 39}, pp.C5-9-C5-10.

\bibitem[Moser(1973)]{Moser}
Moser, J. (1973).
\newblock Stable and Random Motions in Dynamical Systems.
\newblock Annals of mathematical studies, {\bf 77}. Princeton University Press.

\bibitem[Stoffer(1988a)Stoffer]{stoffa}
Stoffer, D. (1988a).
\newblock Transversal homoclinic points and hyperbolic sets for non-autonomous maps {I}.
\newblock {\em J. Appl. Math. and Phys. (ZAMP)\/}, {\bf 39}, 518--549.

\bibitem[Stoffer(1988b)Stoffer]{stoffb}
Stoffer, D. (1988b).
\newblock Transversal homoclinic points and hyperbolic sets for non-autonomous
  maps {II}.
\newblock {\em J. Appl. Math. and Phys. (ZAMP)\/}, {\bf 39}, 783--812.

\bibitem[Wiggins(1999)S.Wiggins]{Wiggins99}
Wiggins S. (1999).
\newblock Chaos in the dynamics generated by sequences of maps, with applications to chaotic advection in flows with 
aperiodic time dependence.
\newblock {\em Z. angew. Math. Phys. (ZAMP)\/}, {\bf 50}, 585--616.

\bibitem[Wiggins(2003)S.Wiggins]{Wiggins03}
Wiggins S. (2003).
\newblock Introduction to Applied Nonlinear Dynamical Systems and Chaos.
\newblock Springer Second Edition.

\bibitem[Wiggins and Mancho (2014)]{wm}
Wiggins, S and Mancho, A. M. (2014).
\newblock Barriers to transport in aperiodically time-dependent two-dimensional velocity fields: Nekhoroshev's theorem and ‘’nearly invariant’’ tori.
\newblock {\em Nonlinear Processes in Geophysics\/}, {\bf 21}, 165--185.

\end{thebibliography}
\end{document}